\documentclass[12pt]{article}
\usepackage{amsfonts,amsthm,amsmath}

\usepackage{amssymb}
\usepackage{color} 
\usepackage{url} 

\numberwithin{equation}{section}

\theoremstyle{plain}
\newtheorem{thm}{Theorem}[section]
\newtheorem{prop}[thm]{Proposition}
\newtheorem{lem}[thm]{Lemma}
\newtheorem{cor}[thm]{Corollary}
\newtheorem{conj}[thm]{Conjecture}

\theoremstyle{definition}
\newtheorem{df}[thm]{Definition}
\newtheorem{rem}[thm]{Remark}
\newtheorem{fact}[thm]{Fact}
\newtheorem{ex}[thm]{Example}
\newtheorem{prob}[thm]{Problem}

\newcommand{\FF}{\mathbb{F}}
\newcommand{\ZZ}{\mathbb{Z}}

\newcommand{\NN}{\mathbb{N}}

\DeclareMathOperator{\wt}{wt}

\DeclareMathOperator{\BG}{G}

\DeclareMathOperator{\ord}{o}
\DeclareMathOperator{\dist}{dist}
\DeclareMathOperator{\diam}{diam}

\begin{document}

\title{The Tutte polynomials of genus $g$}
\author{
Tsuyoshi Miezaki\thanks{Faculty of Science and Engineering, 
Waseda University, 
Tokyo 169-8555, Japan 
miezaki@waseda.jp
},
Manabu Oura\thanks{Graduate School of Natural Science and Technology, 
Kanazawa University,  
Ishikawa 920--1192, Japan 
oura@se.kanazawa-u.ac.jp
},\\
Tadashi Sakuma\thanks{Faculty of Science, 
Yamagata University,  
Yamagata 990--8560, Japan 
sakuma@sci.kj.yamagata-u.ac.jp
}, and 
Hidehiro Shinohara\thanks
{
shino.set.set.set@gmail.com
}
}


\date{}

\maketitle

\begin{abstract}
In the paper [Proceedings of the Japan Academy, Ser. A Mathematical Sciences, 95(10) 111-113],
the authors introduce the concept of the Tutte polynomials of genus $g$ and 
announce that each matroid $M$ can be reconstructed from its Tutte
polynomial of genus $|\mathcal{B}(M)|$, where $\mathcal{B}(M)$ denotes
the family of bases of $M$. In that paper, we also announced that, for all $g$,
there exist inequivalent matroids that have the same Tutte polynomial of genus $g$. 
In this paper, we prove these theorems.
\end{abstract}

\vspace{-0.15in}

\quad


\section{Introduction}\label{sec:Intro}

Let $E$ be a finite set. A \textit{matroid $M$} on $E=E_M$ is a pair 
$(E,\mathcal{I})$, where $\mathcal{I}$ is a non-empty 
family of subsets of $E$ with the following three conditions: 
\begin{itemize}
\item[{\bf{(I1)}}] 
$\emptyset \in \mathcal{I}.$
\item[{\bf{(I2)}}] 
If $I\in \mathcal{I}$ and $J\subset I$, then $J\in \mathcal{I}$;
\item[{\bf{(I3)}}] 
If $I_1,I_2\in \mathcal{I}$ and $|I_1|<|I_2|$, then there exists 
$e\in I_2\setminus I_1$ such that $I_1\cup \{e\}\in \mathcal{I}$. 
\end{itemize}
For a matroid $M$, the set $E_M$ is called a \textit{ground set} of $M$. 
The family $\mathcal{I}$ is called the \textit{independent sets}. 
A matroid $(E, \mathcal{I})$ is \textit{isomorphic} to another matroid 
$(E', \mathcal{I}')$ 
if there is a bijection $\varphi$ from $E$ to $E'$ such that 
$\varphi(I)\in \mathcal{I}'$ holds for each member $I\in \mathcal{I}$, 
and  
$\varphi^{-1}(I')\in \mathcal{I}$ holds for each member $I'\in \mathcal{I}'$. 

It follows from the condition (I3) 
that all maximal independent sets 
in a matroid $M$ have the same cardinality $k$, 
called the \textit{rank} of $M$. The maximal independent sets 
$\mathcal{B}(M)$ are called the \textit{bases} of $M$. 
The \textit{rank} $\rho(A)$ of an arbitrary subset $A$ of $E$ is the cardinality of the largest independent set contained in $A$. 
A \textit{circuit} in a matroid $M$ is a minimal dependent subset of 
$E$.
A \textit{loop} in a matroid $M$ is a singleton $\{e\} \subset E$ such that $\rho(\{e\})=0$,
that is, a circuit of size $1$.
Let us denote the family of circuits of $M$ by $\mathcal{C}(M)$. 
There are many other known axiomatic systems of matroids, all of which are equivalent. 
The following axiomatic system for circuits is also frequently used in this paper.
\begin{itemize}
\item[{\bf{(C1)}}] 
$\emptyset \not\in \mathcal{C}(M).$
\item[{\bf{(C2)}}] 
If $C_1$ and $C_2$ are members of $\mathcal{C}(M)$ and $C_1 \subseteq C_2$, then $C_1=C_2$;
\item[{\bf{(C3)}}] 
If $C_1,C_2 \in \mathcal{C}(M)$, $C_1\neq C_2$, and $e \in C_1 \cap C_2$, then there is 
a member $C_3$ of $\mathcal{C}(M)$ such that $C_3 \subseteq (C_1\cap C_2) \setminus \{e\}$. 
\end{itemize}
The above condition (C3) is called the \textit{circuit elimination axiom}. 

For a matroid $M$, there exists a unique matroid $M^{*}$ such that 
$E_{M^{*}}=E_{M}$, 
and that a set $S\subset E_{M^{*}}$ is independent if and only if 
there exists a basis $B\in \mathcal{B}(M)$ which satisfies 
$S\cap B=\emptyset$. 
Such matroid $M^{*}$ is  called the \textit{dual matroid} of $M$.
The bases of $M^{*}$ is defined by $\{E_M\setminus B | B \in \mathcal{B}(M)\}$,
which is called the \textit{cobases} of $M$. 
It is known that $(M^{*})^{*}=M$.
A \textit{cocircuit} and a \textit{coloop} of a matroid $M$ mean
a circuit and a loop of $M^{*}$, respectively.
The rank of $M^{*}$ is called the \textit{corank} of $M$.

A matroid is \textit{self-dual} if $M^{*}$ is isomorphic to $M$.

We give examples.
\begin{ex}
Let $A$ be a $k\times n$ matrix over a field. 
This gives a matroid $M$ on the set $E=\{1,\ldots,n\}$, in which a set 
$\ell$ is independent if and only if the family of columns of $A$ whose indices belong to $\ell$ is linearly independent. 
Such a matroid is called a vector matroid. 
\end{ex}
\begin{ex}
Let $E=\{1,\ldots,n\}$ and $r$ a non-negative integer. 
We define a matroid on $E$ by taking every $r$-element subset of 
$E$ to be a base. 
This is known as the uniform matroid $U_{r,n}$. 
\end{ex}

A matroid is \textit{empty} if it is isomorphic to $U_{0,0}$;
otherwise it is \textit{non-empty}. 
A matroid $M=(E,\mathcal{I})$ is \textit{separable} if there exists two
non-empty matroids $M_1=(E_1,\mathcal{I}_1)$ and $M_2=(E_2,\mathcal{I}_2)$
with $E_1$ and $E_2$ disjoint such that $E:=E_1\dot{\cup}E_2$ and
$\mathcal{I}:=\{I_1\dot{\cup}I_2 \mid I_1\in\mathcal{I}_1,I_2\in\mathcal{I}_2\}$
hold. In this case, let us denote $M=M_1\dot{\cup}M_2$. 
A matroid is \textit{non-separable} if it is not separable.
It is known that any matroid can be uniquely decomposed into non-separable matroids,
which are called the \textit{components} of the matroid. 

A matroid $M=\dot{\cup}M_i$ is \textit{equivalent} to another matroid 
$M^{\prime}=\dot{\cup}M^{\prime}_i$ if, with reordering if necessary, each $M^{\prime}_{i}$ is
isomorphic to either $M_{i}$ or $M^{\ast}_{i}$. 

Let $M$ be a matroid on the set $E$, having rank function $\rho$. 
The Tutte polynomial of $M$ is defined as follows: 
\[
T(M;x,y)=\sum_{A\subset E}(x-1)^{\rho(E)-\rho(A)}(y-1)^{|A|-\rho(A)}. 
\]
By definition of the dual matroid, $T(M^{*};x,y)=T(M;y,x)$ holds. 

For example, the Tutte polynomial of the uniform matroid $U_{r,n}$ is
\[
T(U_{r,n};x,y) =
\sum_{i=0}^r
\binom{n}{i}(x-1)^{r-i} +
\sum_{i=r+1}^n
\binom{n}{i}(y-1)^{i-r}.
\]
It is well known that, from the Tutte polynomial of a matroid $M$, we can read out 
all of the rank, the corank, the number of loops, the number of the coloops,
the number of bases (cobases) of $M$ (\cite{Welsh}).

It is easy to show that $T(M;x,y)$ is a matroid invariant and 
if two matroids have the same Tutte polynomial, 
we call them \textit{$T$-equivalent}.
It is also known that there exist two non-isomorphic matroids, 
which are $T$-equivalent (\cite{Welsh}).

This gives rise to a natural question: is there
a generalization of the Tutte polynomial?
We now present the concept of the Tutte polynomial of genus $g$. 
(This definition already appeared in the announcement paper \cite{MOSS}.)
\begin{df}
Let $M:=(E,\mathcal{I})$ be a matroid. Let $(A_i\mid i=1,\ldots,g)$ be an
arbitrary $g$-tuple of subsets of $E$. Let {$\Lambda_1:=\{1,\ldots,g\}$} and 
let {$\Lambda_2:=\displaystyle{{\Lambda_1} \choose 2}$}. For every element $\lambda \in {\Lambda_2}$,
let us denote {$A_{\cap(\lambda)}:=\cap_{i\in\lambda}A_i$} and {$A_{\cup(\lambda)}:=\cup_{i\in\lambda}A_i$}.
By using these notations, let us define
\begin{eqnarray*}
T^{(g)}(A_1,\ldots,A_g)&:=&
\prod_{\lambda \in {\Lambda_1}}(x_{\lambda}-1)^{\rho(E)-\rho(A_{\lambda})}(y_{\lambda}-1)^{|A_{\lambda}|-\rho(A_{\lambda})}\\
&&\times \prod_{\lambda \in {\Lambda_2}}(x_{\cap(\lambda)}-1)^{\rho(E)-\rho(A_{\cap(\lambda)})}(y_{\cap(\lambda)}-1)^{|A_{\cap(\lambda)}|-\rho(A_{\cap(\lambda)})}\\
&&\times \prod_{\lambda \in {\Lambda_2}}(x_{\cup(\lambda)}-1)^{\rho(E)-\rho(A_{\cup(\lambda)})}(y_{\cup(\lambda)}-1)^{|A_{\cup(\lambda)}|-\rho(A_{\cup(\lambda)})}
\end{eqnarray*}
Then the genus $g$ Tutte polynomial $T^{(g)}(M)$ of the matroid $M$ will be
defined as follows:
\begin{align}
T^{(g)}(M)
&:=T^{(g)}(M;x_{\lambda_1},y_{\lambda_1},x_{\cap\lambda_2},y_{\cap\lambda_2},x_{\cup\lambda_2},y_{\cup\lambda_2}
:\lambda_1\in \Lambda_1,\lambda_2\in \Lambda_2) \label{definition_genus_Tutte} \\
&:=\sum_{A_1,\ldots,A_n \subseteq E}T^{(g)}(A_1,\ldots,A_g)\nonumber
\end{align}
\end{df}

\noindent
For any positive integer $g$, we have 
$$T^{(g)}(M)=\frac{1}{2^{|E_M|}}T^{(g+1)}(M)|_{x_{A_{g+1}}=2,y_{A_{g+1}}=2,x_{A_j\cap A_{g+1}}=2 (j=1,\ldots,g),y_{A_k\cup A_{g+1}}=2 (k=1,\ldots,g)}.$$
It is straightforward to see that 
$T^{(g)}$ is a matroid invariant.
For example, 
the genus $2$ Tutte polynomial $T^{(2)}(M)$ of the matroid $M$ is 
as follows: 
\begin{equation*}
  \begin{split}
&T(M;x_1, x_2, x_{\cap\{1,2\}}, x_{\cup\{1,2\}}, y_1, y_2, y_{\cap\{1,2\}}, y_{\cup\{1,2\}})\\
&=\sum_{A_1,A_2\subset E}
(x_1-1)^{\rho(E)-\rho(A_1)}
(x_2-1)^{\rho(E)-\rho(A_2)}\\
&\qquad \qquad ~~
(x_{\cap\{1,2\}}-1)^{\rho(E)-\rho(A_1\cap A_2)}
(x_{\cup\{1,2\}}-1)^{\rho(E)-\rho(A_1\cup A_2)}\\
&\qquad \qquad ~~
(y_1-1)^{|A_1|-\rho(A_1)}
(y_2-1)^{|A_2|-\rho(A_2)}\\
&\qquad \qquad ~~
(y_{\cap\{1,2\}}-1)^{|A_1\cap A_2|-\rho(A_1\cap A_2)}
(y_{\cup\{1,2\}}-1)^{|A_1\cup A_2|-\rho(A_1\cup A_2)}.
\end{split}
\end{equation*}

The Whitney rank generating function is defined as follows:
\[
R(M;x,y)=\sum_{A\subset E}x^{\rho(E)-\rho(A)}y^{|A|-\rho(A)}. 
\]
For every matroid $M$, its Tutte polynomial and its Whitney rank
generating function are equivalent under a simple change of variables:
\[
R(M;x,y)=T(M;x+1,y+1).
\]
We also naturally define the genus $g$ Whitney rank generating
function $R^{(g)}(M)$ as follows:
\begin{df}
We will reuse the notations in the definition of the Tutte polynomial
of genus $g$. By using these notations, let us define
\begin{align*}
R^{(g)}(M)
&:=R^{(g)}(M;x_{\lambda_1},y_{\lambda_1},x_{\cap\lambda_2},y_{\cap\lambda_2},x_{\cup\lambda_2},y_{\cup\lambda_2}
:\lambda_1\in \Lambda_1,\lambda_2\in \Lambda_2)\\
&=T^{(g)}(M;x_{\lambda_1}+1,y_{\lambda_1}+1,x_{\cap\lambda_2}+1,y_{\cap\lambda_2}+1,x_{\cup\lambda_2}+1,y_{\cup\lambda_2}+1
:\lambda_1\in \Lambda_1,\lambda_2\in \Lambda_2).
\end{align*}
For every $g$-tuple $(A_i\mid i=1,\ldots,g)$ of subsets of $E_M$, 
$R^{(g)}(A_1,\ldots,A_g)$ also can be defined naturally.
\end{df}
It is clear that $R^{(g)}(M^{*};x,y)=R^{(g)}(M;y,x)$ holds for every positive integer $g$, again. 

To state our results, 
we introduce the following concepts. 
If two matroids have the same Tutte polynomial for genus $g$, 
we call them $T^{(g)}$-equivalent. 
More generally, we introduce the following: 
\begin{df}
Let $\mathcal{M}$ be the set of all matroids. 
For $X\subset \mathcal{M}$, we call $X$ is {\sl of class $T^{(g)}$} if 
$T^{(g)}$ is a complete invariant for $X$. 
\end{df}
It is not difficult to see that each matroid $M=(E,\mathcal{I})$ can be reconstructed from its Tutte polynomial
of genus $|E|+|\mathcal{B}(M)|$. Hence the set of Tutte polynomials {$\{T^{(g)}\}_{g=1}^{\infty}$} is
a complete invariant for matroids.

Here we state the following problem: 
\begin{prob}
For $X\subset \mathcal{M}$, 
determine an appropriate small upper bound $g$ such that $X$ is of class $T^{(g)}$. 
\end{prob}

In the paper \cite{MOSS}, the authors announced that each matroid $M$ can be
reconstructed from its Tutte polynomial of genus $|\mathcal{B}(M)|$ and that, 
for all $g$, there exist inequivalent matroids that have the same Tutte
polynomial of genus $g$. 
For a real number $r$, $\lceil r \rceil$ denotes 
the minimum integer at least $r$. 
In this paper, we prove these statements.
\begin{thm}\label{thm:main}
\begin{enumerate}
\item [{\rm (1)}]
Every matroid $M$ can be reconstructed from its Tutte polynomial of genus $|\mathcal{B}(M)|$. 
Especially, for $g\in \NN$, the set of matroids 
$\{M\in \mathcal{M}\mid |\mathcal{B}(M)|\leq g\}$
is of class $T^{(g)}$. 
\item [{\rm (2)}]
For any positive integer $g$, there exist two 
matroids $M$ and $N$ such that
$T^{(g)}(M)=T^{(g)}(N)$ and 
$T^{(\lceil\sqrt{2}g \rceil)}(M) \neq T^{(\lceil\sqrt{2}g \rceil)}(N)$.
\end{enumerate}
\end{thm}

This paper is organized as follows. 
In Section \ref{sec:2}, we prepare some preliminaries.
In Sections \ref{sec:3}, \ref{sec:4} we give a proof of Theorem \ref{thm:main} (1) and (2), respectively. 
In section \ref{sec:rem}, we give some remarks and open problems.

All computations presented in this paper were performed using 
{\sc Magma} \cite{Magma}
and 
{\sc Mathematica} \cite{Mathematica}.

\section{Preliminaries}\label{sec:2}

\subsection{Elementary Properties on Matroids}

\begin{df}
A matroid $M$ is $T^{(g)}$-unique if $M$ is a 
unique matroid, up to isomorphism, whose Tutte 
polynomial of genus $g$ is $T^{(g)}(M)$.
\end{df}

\begin{lem}\label{lem:circuit}
If a matroid $M$ has no circuit $C$ with 
$ 3\leq |C|$
then $M$ is 
a direct sum of matroids of rank at most one. 
In particular, $M$ is $T^{(1)}$-unique. 
\end{lem}

\begin{proof}
Suppose that a matroid $M$ has no circuit $C$ with $ 3\leq |C|$. 
If $e$ is a loop or a coloop of $M$, then $M$ can be decomposed into the direct sum of $M\setminus\{e\}$ and $\{e\}$.
Now we assume here that $M$ has neither a loop nor a coloop.
Then every circuit $C$ of $M$ satisfies $|C|=2$.
In this case, every element $e$ of $E_M$ is contained in some circuit of two elements,
and if $e$ is contained in two distinct circuits $C_1=\{e,x\}$ and $C_2=\{e,y\}$
then the set $\{x,y\}$ is also a circuit of $M$. 
Hence if we denote $p \sim q$ for two elements $p,q$ of $E_M$ when $\{p,q\}$ is a circuit of $M$,
then this relation $\sim$ is an equivalence relation on $E_M$.
Each equivalence class of the above relation $\sim$ on $E_M$ is corresponding to a uniform matroid of rank $1$
and $M$ can be decomposed into the direct sum of these uniform matroids. 
\end{proof}

\begin{lem}[{\cite[Proposition 1.1.6]{Oxley}}]\label{lem:base+one}
Let $M=(E,\mathcal{I})$ be a matroid, $I\subseteq \mathcal{I}$ be an independent set of $M$,
$e \in E\setminus I$ be an element outside of $I$. If $\rho(I\cup\{e\})=|I|$ holds,
then $I\cup\{e\}$ contains a unique circuit of $M$.
\end{lem}

\begin{lem}\label{lem:distance1}
Let $C$ be a circuit of $M$ with $|C|\geq 3$, and let $B$ be a base of $M$
such that $|C\setminus B|=1$ holds. Then the following two conditions hold.
\begin{enumerate}
  \item The family of bases contained in $B \cup C$ is 
        $\mathcal{F}:=\{(B \cup C) \setminus \{p\} | p \in C\}$.
  \item If a base $B^{\prime}$ of $M$ satisfies the condition that
        $\forall B \in \mathcal{F}, |B \setminus B^{\prime}| \leq 1$, 
        then $B^{\prime}$ is identical to a base in $\mathcal{F}$.
\end{enumerate}
\end{lem}

\begin{proof}
The first condition follows immediately from Lemma \ref{lem:base+one}. 
For the second condition, it is clear that $|C \setminus B^{\prime}|=1$, for otherwise
there exist two distinct elements $u,v \in C \setminus B$ such that
$\exists p\in C\setminus\{u,v\}, |B^{\prime} \setminus ((B\cup C)\setminus\{p\})| \geq 2$, a contradiction.
Then the second condition follows from the first condition. 
\end{proof}

\begin{lem}\label{lem:every_base_has_a_3-circuit}
If a matroid $M$ has a circuit of at least three elements, 
then, for every base $B$ of $M$, there exists a circuit $C$ with $|C| \geq 3$
such that $|C\setminus B|=1$. 
\end{lem}

\begin{proof}
Suppose not. Then, there exists a base $B$ such that, for any element $e\in E_M\setminus B$,
$\{e\}\cup B$ has a unique $2$-circuit $\{e,e'\}$. Let $C=\{e_1,e_2,\ldots,e_m\}$ be a circuit
of $M$of at least  $3$ elements. Let $e'_i$ denote an element of $B$ such that either
$\{e_i,e'_i\}$ forms a $2$-circuit of $M$ or $e'_i=e_i$. Since the relation that two
elements of $E_M$ form a $2$-circuit (are parallel) is an equivalence relation on $E_M$,
the set $\{e'_1,e'_2,\ldots,e'_m\}$ is also a circuit of $M$of at least  $3$ elements.
Nevertheless, the above circuit is a subset of $B$, a contradiction. 
\end{proof}



\begin{lem}\label{lem:|B-C|>0}
Let $B$ be a base of a non-separable matroid $M$ and let $C$ be a circuit of $M$
such that $|C\setminus B|=1$. If $|B\setminus C|>0$ then $M$ has another base
$B'$ such that $|C \setminus B'|=1$ and $|B \setminus B'|\geq 2$ hold.
\end{lem}

\begin{proof}
If $|B\setminus C|>0$ then, for every element $e \in B\setminus C$,
$(E_M\setminus B)\cup\{e\}$ contains a unique cocircuit of at least $2$ elements, 
and hence, there exists an element $e' \in E_M\setminus B$ such that 
$B'':=(B\setminus\{e\})\cup\{e'\}$ is another base of $M$ and $|C \setminus B''|=1$
also holds. Then, for every element $e''\in C \cap B'' \neq \emptyset$,
$B':=(B''\setminus\{e''\})\cup(C\setminus B'')$ is a base of $M$ 
such that $|C \setminus B'|=1$ and $B \setminus B'=\{e,e''\}$ hold.
\end{proof}





\subsection{Matroid Base Graph}

First, let us define the matroid base graph of a matroid.
\begin{df}[\cite{{HNT},{Maurer1},{Maurer2}}]
Let $M$ be a matroid, and let $\mathcal{B}(M)$ be the bases of $M$.
Then the {\it matroid base graph} $\BG(M)=(V(\BG(M)),E(\BG(M)))$ of $M$
is defined by $V(\BG(M)):=\mathcal{B}(M)$ and
$E(\BG(M)):=\{\{B,B^{\prime}\}:B,B^{\prime} \in \mathcal{B}(M), |B \setminus {B^{\prime}}|=1\}$.
\end{df}

For a vertex $v$ of a graph $G$, its {\it open neighborhood} $N_G(v)$ is defined by the set $\{u\in V(G): \{u,v\} \in E(G)\}$, 
and its {\it closed neighborhood} $N_G[v]$ is defined by the set $N_G(v) \cup \{v\}$. 

\begin{lem}[{\cite[Lemma 4.1]{HNT}}]\label{pi}
Let $M$ be a non-empty matroid and $B$ a base of $M$. 
Then there are two partitions $\pi$ and $\pi'$ of $N_{\BG(M)}(B)$
into non-void subsets so that 
\begin{enumerate}
\item 
two vertices $B_1,B_2$ are adjacent in $N_{\BG(M)}(B)$ if and only if 
$B_1$ and $B_2$ are in the same equivalence class of $\pi$ or $\pi'$, and 
\item 
if $p\in \pi$ and $q\in \pi'$, then $|p\cap q|\leq 1$. 
\end{enumerate}
Further, if $M$ is non-separable, the pair of partitions 
$\pi$ and $\pi'$ are unique, up to order. 
\end{lem}
For any element $x\in B$, let $p^{B}_x:=\{B'\in N_{\BG(M)}(B)\mid x\not\in B'\}$ and 
for any element $y\in E\setminus B$, 
let $q^{B}_y:=\{B'\in N_{\BG(M)}(B)\mid y\in B'\}$. 
Then both 
$\pi =\{p^{B}_x\mid x\in B\mbox{ and }p^{B}_x\neq \emptyset\}$ and 
$\pi' =\{q^{B}_y\mid y\in E\setminus B\mbox{ and }q^{B}_y\neq \emptyset\}$ are 
the partitions with the desired properties mentioned in Lemma \ref{pi}. 
We will denote the former partition by $\pi(B,M)$ and
the latter partition by $\pi'(B,M)$.

Combining the above fact with Lemma \ref{pi}, we have the following:

\begin{lem}[\cite{HNT}]\label{lem:detect_N(B)}
Let $M$ be a matroid, $B$ an arbitrary vertex of $\BG(M)$. 
Suppose that we do not know the whole graph $\BG(M)$, while we know only 
its subgraph induced by $N_{\BG(M)}[B]$. Furthermore, suppose that we also know
each of the two partitions of $N_{\BG(M)}(B)$ corresponding to $\pi(B,M)$ and
$\pi'(B,M)$. Then we can construct a labeling on the bases corresponding to
the vertices $N_{\BG(M)}[B]$, that is, we can recover the bases of $M$
corresponding to $N_{\BG(M)}[B]$.  
\end{lem}

The above lemma is essentially proven in the paper \cite{HNT}. Here
we give its proof, only for the convenience of readers. 
\begin{proof}
Let us treat each vertex in $N_{\BG(M)}[B]$ as a base of $M$. 
First let us set $B:=\{x_1,\ldots,x_{\rho(M)}\}$ and $E_M\setminus B:=\{y_{1}, \ldots,y_{\rho(M^{*})}\}$. 
From Lemma \ref{lem:distance1}, we have that, for each element $y \in E_M\setminus B$,
the subset $B\cup(\cup_{B'\in q^{B}_y} B')$ of $E_M$ turns out to be $B\cup\{y\}$. Furthermore,
from Lemma \ref{lem:circuit}, this subset $B\cup\{y\}=B\cup(\cup_{B'\in q^{B}_y} B')$ of $E_M$
contains a unique circuit of $M$. Let us denote this circuit by $C_y$. In the same way,
for every element $x \in B$, let us denote the unique cocircuit in $(E_M\setminus B) \cup \{x\}$
by $D_x$. For every vertex $B''$ in $N_{\BG(M)}(B)$, we can find a unique pair of elements
$(x(B''),y(B'')) \in \{(x,y) | x \in B, y \in E_M \setminus B\}$ such that
$\{B''\}=p^{B}_{x(B'')}\cap q^{B}_{y(B'')}$. This fact reveals that $B''=(B\setminus\{x(B'')\})\cup\{y(B'')\}$. 
We also know that $\{x(B''),y(B'')\}=D_{x(B'')} \cap C_{y(B'')}$.
Hence we can fix the labels on all the bases $N_{\BG(M)}[B]$.  
In addition, by using the labels on the bases $N_{\BG(M)}[B]$, 
we can also fix all the labels on the circuits $C_y (\forall y \in E_M\setminus B)$ and
the cocircuits $D_x (\forall x \in B)$ of $M$. 
\end{proof}


In general, matroids are not necessarily determined from their 
matroid base graphs. For example, every matroid $M$ and its dual $M^*$ have
their common matroid base graph. More precisely, it is known that the matroid
base graphs $\BG(M_1)$ and $\BG(M_2)$ are isomorphic as a graph if and only if
$M_1$ and $M_2$ are equivalent (\cite[Theorem 5.3]{HNT}). Nevertheless, the following
facts are known.

Let $p,q$ are two vertices of a graph $G$, then $\dist(p,q)$ denotes the distance
of $p$ and $q$ in $G$. 

\begin{lem}[{\cite[Lemma 3.2]{HNT}}]\label{lem:common_neighbor}
Let $M$ be a matroid. 
If $B$ and $B^{\prime}$ are two vertices of $\BG(M)$ such that $\dist(B,B^{\prime})=2$,
then $N_{\BG(M)}(B) \cap N_{\BG(M)}(B^{\prime})$ is an induced subgraph of the cycle of 
four vertices and contains two nonadjacent vertices. That is, the intersection of the
open neighborhoods is either two isolated vertices, a path of three vertices, or
a cycle of four vertices.
\end{lem}

\begin{lem}[{\cite[Corollary 3.2.1]{HNT}}]\label{lem:label_fixed}
Let $M$ be a matroid and $B, B^{\prime}$ be two vertices of $\BG(M)$ such that
$\dist(B,B^{\prime})=2$. Let $B_1$ and $B_2$ be nonadjacent vertices of
$N_{\BG(M)}(B) \cap N_{\BG(M)}(B^{\prime})$.
Then if the labels on $B,B_1$ and $B_2$ are known, the labels on $B^{\prime}$
can be determined by $B^{\prime}=(B_1 \cap B_2) \cup (B_1\setminus B) \cup (B_2 \setminus B)$.
\end{lem}

By using Lemmata \ref{pi}, \ref{lem:detect_N(B)}, \ref{lem:common_neighbor}, and \ref{lem:label_fixed}, 
C.A.Holzmann, P.G. Norton and M.D. Tobey \cite{HNT}
prove the following: 

\begin{lem}[{\cite[Corollary 3.2.2]{HNT}}]\label{lem:recover}
Let $M$ be a matroid, and let $B$ be a vertex of $\BG(M)$.
For every nonnegative integer $i$, let us define
$V_B(i):=\{B'\in V(\BG(M)) | \dist(B,B') \leq i\}$.
Suppose that we do not know the whole graph $\BG(M)$, while we know only its
subgraph $\BG(M)[V_B(i)]$ for some nonnegative integer $i$. Furthermore,
suppose that we also know each of the two partitions of $N_{\BG(M)}(B)$
corresponding to $\pi(B,M)$ and $\pi'(B,M)$. Let us set $B:=\{x_1,\ldots,x_{\rho(M)}\}$
and $E_M\setminus B:=\{y_{1}, \ldots,y_{\rho(M^{*})}\}$. In this case, the labels on all
vertices of $\BG(M)[V_B(i)]$ can be determined. 
\end{lem}


\section{Proof of Theorem \ref{thm:main} (1)}\label{sec:3}

We note that the number of the loops and the coloops is known from 
$T^{(1)}$. So we assume that there exists neither a loop nor a coloop. 
Furthermore, we also assume that our matroid $M$ is non-separable simply
because each $T^{(g)}$ (resp. $R^{(g)}$) factorizes into a product of
the genus $g$ Tutte polynomials (resp. the genus $g$ Whitney rank
generating functions) of non-separable matroids. 
Now we show the following statement equivalent to Theorem \ref{thm:main}:
\begin{thm}\label{genus_base_number}
Every non-separable matroid $M$ can be reconstructed from its Whitney rank generating
function of genus $|\mathcal{B}(M)|$. 
\end{thm}
\begin{proof}

First, we show that the Whitney rank generating function
of genus $|\mathcal{B}(M)|$ constructs the matroid bases graph $\BG(M)$.  
Then we show that the Whitney 
rank generating function of the genus
$|\mathcal{B}(M)|$ reconstructs the matroid $M$ itself. 

\begin{enumerate}
\item 
Let $M=(E,\mathcal{I})$ be a matroid with $g$ bases. We claim that 
$R^{(g)}(M)$ constructs $\BG(M)$. 

In this case, $R^{(g)}(M)$ has a monomial $R^{(g)}(A_1,\ldots,A_g)$ 
such that each set $A_i(i=1,\ldots,g)$ corresponds
one-to-one to each base of $M$ and it can be detected as follows.
For a monomial of $R^{(g)}(M)$, it is clear that each set $A_i$
is a base of $M$ if and only if the variables $x_{A_i}$ and $y_{A_i}$
vanish (their exponents are zero) in the monomial.
In addition, it is also clear that two bases $A_i, A_j (i\neq j)$
are mutually distinct if and only if the exponent of the variable
$x_{A_i \cap A_j}$ is positive. 
By using this monomial $R^{(g)}(A_1,\ldots,A_g)$, we can obtain
a graph $G=(V(G),E(G))$ isomorphic to the matroid base graph $\BG(M)$
such that $V(G):=\{A_1, \ldots, A_g\}$ and
$E(G):=\{\{A_i,A_j\} | \textrm{ the exponent of } x_{A_i \cap A_j}
\textrm{ in } R^{(g)}(A_1,\ldots,A_g) \textrm{ is } 1\}$.

\item 
Let $M=(E,\mathcal{I})$ be a non-separable matroid. 
As we mentioned in Section \ref{sec:Intro}, 
the Tutte polynomial $T^{(1)}(M)$ gives the 
following numbers: 
\begin{itemize}
\item 
$\nu:=$ $|E|$
\item 
$\rho:=$ the rank of $M$
\item 
$\chi:=$ the corank of $M$
\item
$|\mathcal{B}(M)|$, the number of bases of $M$.
\end{itemize}
By Lemma \ref{lem:circuit}, we can assume that $M$ has a circuit of at least 
three elements. The matroid $M$ has two distinct bases $B_1,B_2$ such that
the distance $\dist(B_1,B_2)$ between these two vertices $B_1,B_2$ of $\BG(M)$
coincides with the diameter $\diam(\BG(M))$ of the graph $\BG(M)$. 
Here we can assume that the number $\diam(\BG(M))$ is at least $2$, for otherwise, 
combining Lemmata \ref{lem:distance1} and \ref{lem:every_base_has_a_3-circuit}, 
we have $\chi=1$, and $M$ turns out to be a single circuit, that is,
$M=U_{k-1,k}$ for some positive integer $k$, which is $T^{(1)}$-unique.
Hence we assume that $\dist(B_1,B_2)\geq 2$.  


From Lemma \ref{lem:every_base_has_a_3-circuit}, there exists a circuits $C_1$ of $M$
such that $|C_1|\geq 3$ and $|C_1\setminus B_1|=1$ hold.
Let $c_1$ be the unique element in $C_1\setminus B_1$. 

Let $g:=|\mathcal{B}(M)|$.
Then $R^{(g)}(M)$ has a monomial $R^{(g)}(A_1,\ldots,A_g)$ 
such that $A_1=C_1$, $A_2=B_1$, and 
$\{A_3,\ldots,A_{|C_1|+1}\}=q^{B_1}_{c_1}:=\{B\in N_{\BG(M)}(B_1)\mid c_1 \in B\}$
and $\{A_2,\ldots,A_g\}=\mathcal{B}(M)\setminus\{B_2\}$ hold.
We guess this monomial $R^{(g)}(A_1,\ldots,A_g)$ of $R^{(g)}(M)$.
That is, we take each monomial of $R^{(g)}(M)$ 
one by one, consider the monomial as if it were 
the monomial $R^{(g)}(A_1,\ldots,A_g)$, and 
check whether this candidate satisfies all of 
the properties and conditions described
in the proofs that follow. If we find that the candidate at hand does not satisfy even one
of these properties and conditions, we discard it. 

Since $|C_1|>2$, we have $A_1 \cup A_2=A_2\cup A_3=A_3 \cup A_4=A_4 \cup A_2=\cup_{2\leq i \leq |C_1|+1 } A_i$. 

Summarizing, we obtain $A_1, \ldots, A_{g}$ with a positive integer $m(\geq 3)$ such that 
\begin{itemize} 
  \item $A_{2}, \ldots, A_{g}$ are mutually distinct bases of $M$.
We can confirm this property by checking that the exponents of the variables
$x_{A_i}, y_{A_i} (i=2,\ldots,g)$ are zero and the exponents of the variables
$x_{A_i \cap A_j} (2 \leq i < j \leq g)$ are positive in our candidate monomial. 
\item $A_1$ is a circuit of $M$ of $m$ elements such that 
$|A_1 \setminus A_i|=1 (i=2,\ldots,m)$ hold. 
We can confirm this property by checking that,
for $2 \leq i \leq m$, the exponents of the variables
$x_{A_1}$ and $x_{A_1 \cap A_i}$ are all $\rho-m+1$,
the exponent of the variable $y_{A_1}$ is $1$, and
the exponents of the variables $y_{A_1 \cap A_i}$ are all $0$
(c.f. Lemma \ref{lem:distance1}) in our candidate monomial. 
  \item Let $H$ be a graph such that $V(H):=\{A_i| 2 \leq i \leq g \}$ and
$E(H):=\{\{A_i,A_j\} | \textrm{ the exponent of } x_{A_i \cap A_j} \textrm{ in } R^{(g)}(A_1,\ldots,A_g) \textrm{ is } 1\}$.
Let \\ $\BG(M)-B_2$ denote the resulting subgraph obtained by deleting the vertex $B_2$ from $\BG(M)$.
Then there exists a graph isomorphism $f: V(\BG(M)-B_2) \to V(H)$ such that $f(B_1)=A_2$ holds.
\end{itemize} 

Let $S$ denote the unique missing base in $\mathcal{B}(M)\setminus \{A_2,\ldots,A_g\}$.

From now on, we will show that 
we can reconstruct a correct labeling on all the bases $\mathcal{B}(M)$. 


Since our graph $H$ is isomorphic to the subgraph $\BG(M)-B_2$, 
we can treat this graph $H$ as if it is $\BG(M)-B_2$. 
From Lemma \ref{pi}, there is a unique pair $\{\pi, \pi'\}$ of partitions of 
$N_{H}(A_2)$ and only one of the two partitions (that is, $\pi'(A_2,M)$)
has the equivalence class $\{A_3,\ldots,A_{m+1}\}$. Thus we can detect which of
the two partitions of $N_{H}(A_2)$ is $\pi'(A_2,M)$. 
Combining the above with Lemmata \ref{lem:detect_N(B)}, \ref{lem:common_neighbor},
\ref{lem:label_fixed} and \ref{lem:recover}, we can recover the labels
on all bases $A_2,\ldots,A_g$ up to permutation of indices.


From now on, we will detect the labels on the unique missing base $S$ of $M$ on the outside of $V(H)$.
From Lemma \ref{lem:every_base_has_a_3-circuit}, $M$ has a circuit $C$ of at least $3$ elements such that $|B \cup C|=\rho +1$ holds.
Since $|C|>2$, from Lemma \ref{lem:distance1}, $B \cup C$ contains (at least) $3$ mutually distinct bases, namely, $S, A_{\sigma(1)}$ and $ A_{\sigma(2)}$.
In order to detect the labels on the circuit $C$, first we find candidates of the pair 
$\{A_{\sigma(1)}, A_{\sigma(2)}\}$ of distinct bases different from $S$, as follows:
By using the above labeling, we enumerate the families $\{\mathcal{X}_i\}$ of labeled bases
such that $\mathcal{X}_i=\{A_{\sigma_i(1)},A_{\sigma_i(2)},\ldots,A_{\sigma_i(k_i)}\}$,  
$|\cup_{j \leq k_i} A_{\sigma_i(j)}|=\rho+1$, 
$\forall A_i \in \{A_2,\ldots,A_g\} \setminus \{A_{\sigma_i(1)},\ldots,A_{\sigma_i(k_i)}\},
|A_i \cup (\cup_{j \leq k_i} A_{\sigma_i(j)})|\geq \rho+2$. 
That is, $\mathcal{X}_i$ denotes the family of labeled
bases of $V(H)$ included in
the set $A_{\sigma_i(1)} \cup A_{\sigma_i(2)}$. 
Let $C(\mathcal{X}_i):=\cup_{j \leq k_i} ((A_{\sigma_i(1)}\cup A_{\sigma_i(2)}) \setminus A_{\sigma_i(j)})$.
Then either (1): $M$ has a circuit $C'$of at least  $3$ elements such that
$C(\mathcal{X}_i)=C'\cap S$ holds, or
(2): The set $C(\mathcal{X}_i)$ is a circuit of $M$ of $k_i$ elements.

In the list of families $\{\mathcal{X}_i\}$, we can find a family 
$\mathcal{X}_j$ in the former case (1), as follows:
$M$ has a circuit $C'$ of at least $3$ elements that satisfies $C(\mathcal{X}_j)=C'\cap S$,
if and only if one of the following two mutually exclusive conditions holds.
\begin{enumerate}
  \item{If the set $A_{\sigma_j(1)}\cup A_{\sigma_j(2)}$ is a circuit of $M$,
  we can verify this fact by checking that the three conditions 
  $|A_{\sigma_j(1)}\cup A_{\sigma_j(2)}|=k_j+1$, $|\cap_{i=1}^{k_j} A_{\sigma_j(i)}|=1$ and $k_j=\rho$ hold.  
  In this case, we have that 
  $S=(A_{\sigma_j(1)}\cup A_{\sigma_j(2)})\setminus(\cap_{i=1}^{k_j} A_{\sigma_j(i)})$.}
  \item{Otherwise, from Lemmata \ref{lem:distance1} and \ref{lem:|B-C|>0},
  there exists $\mathcal{X}_{\ell} (\ell \neq j)$ such that
  $C(\mathcal{X}_j) \subsetneq C(\mathcal{X}_{\ell})$ and 
  $|C(\mathcal{X}_{\ell}) \setminus C(\mathcal{X}_{j})|=1$.
  In this case, we have that
  $S=(A_{\sigma_j(1)}\cup A_{\sigma_j(2)}) \setminus (C(\mathcal{X}_{\ell}) \setminus C(\mathcal{X}_j))$.}
\end{enumerate}
Hence, in either case, we can detect the labels on the missing base $S$.
This completes the proof.
\end{enumerate}
\end{proof}

\section{Proof of Theorem \ref{thm:main} (2)}\label{sec:4}

\subsection{Matroids $M,N$ such that $T^{(1)}(M)=T^{(1)}(N)$ and $T^{(2)}(M)\neq T^{(2)}(N)$}\label{subsec:4.1}

In \cite{Kahn}, the concept of ``matroid relaxation" was given. 
Given a matroid $M = (E,\mathcal{B})$ with a subset 
$X \subset E$ that is both a circuit
and a hyperplane, 
we can define a new matroid $M'$ as the matroid with basis
$\mathcal{B}'= \mathcal{B}\cup \{X\}$. That $M'$ is indeed a matroid 
is easy to check. The Tutte polynomial of $M'$ can
be computed easily from the Tutte polynomial of $M$ by 
\begin{align} 
T(M';x,y) = T(M;x,y) - xy + x + y. \label{rexlation} 
\end{align} 

In this subsection, $E$ denote the set $\{ z\mid z\in \ZZ, 1\leq z \leq 2n\}$. 
For an independence system $M$, 
let $\mathcal{B}(M)$ denote the set of maximal independent  set. 
In this subsection, let us define 
$X_1:=\{ z\mid z\in \ZZ, 1\leq z \leq n\}$, 
$X_2:=\{ z\mid z\in \ZZ, n+1\leq z \leq 2n\}$, 
$X_3:=(X_2\cup \{ n\}) \setminus \{ 2n\}$. 

Let $R_{2n}$ denote the independence system on $E$ such that  
$\mathcal{B}(R_{2n})=\mathcal{B}(U_{n, 2n}) 
\setminus \{ X_1, X_2\}$,
and let $Q_{2n}$ denote the independence system on $E$ such that  
$\mathcal{B}(Q_{2n})= \mathcal{B}(U_{n, 2n})\setminus \{ X_1, X_3 \}$.

Then, the following two facts can be easily proved. 
\begin{fact} 
For any positive integer $n$, $R_{2n}$ is a matroid. 
\end{fact} 
\begin{fact} 
The independent system $Q_{4}$ is not a matroid since the application of the independent set exchange property to 
independent sets $\{ \{ 1,3\}, \{ 2\}\}$
forces one of $\{ 1, 2\}$ and $\{ 2, 3\}$ to be independent. 
However, $Q_{2n}$ is a matroid if $n\geq 3$. 
Note that the dual matroid of $R_{2n}$ is itself, 
and the dual matroid of $Q_{2n}$ is isomorphic to $Q_{2n}$. 
\end{fact} 
Therefore, $n\geq 3$ is a natural assumption in the following propositions.  

\begin{prop} 
For any integer $n\geq 3$, $T^{(1)}(R_{2n})
=T^{(1)}(Q_{2n})$ holds. 
\end{prop} 
\begin{proof} 
Since we obtain $U_{n, 2n}$ by two relaxations from both of $R_{2n}$ and $Q_{2n}$, 
Therefore, the formula (\ref{rexlation}) yields 
\begin{align*} 
 &\, \, \quad T^{(1)}(R_{2n}; x, y)\\ 
&=T^{(1)}(Q_{2n}; x, y) \\ 
&=T^{(1)}(U_{n, 2n})+2(x y-x-y) \\ 
&=\sum_{i=0}^n
\binom{2n}{i}(x-1)^{n-i} +
\sum_{i=n+1}^{2n}
\binom{2n}{i}(y-1)^{i-n}+2(x y-x-y). \qedhere 
\end{align*} 
\end{proof}

\begin{prop} 
For any integer $n\geq 3$, $T^{(2)}(R_{2n})\not=T^{(2)}(Q_{2n})$ holds. 
\end{prop} 
\begin{proof}
Let $(A,B)$ be a pair of sets such that $A \cap B=\emptyset, |A|=|B|=n$ and $\rho(A)=\rho(B)=n-1$ hold.
The matroid $R_{2n}$ has such a pair while $Q_{2n}$ does not. 
Hence the Tutte polynomial $T^{(2)}(R_{2n})$ have the term 
\begin{align*} 
(x_1-1)^{n-1}(y_1-1)(x_2-1)^{n-1}(y_2-1)(x_{\cap(1,2)}-1)^{n}(y_{\cup(1,2)}-1)^{n}, 
\end{align*} 
which is generated by $(A, B)=(X_1, X_2)$. 
Since $T^{(2)}(Q_{2n})$ does not have this term, we have $T^{(2)}(R_{2n})\not=T^{(2)}(Q_{2n})$.  
\end{proof} 
These two propositions prove Theorem 1. 1 (2). 
\begin{rem} 
We can calculate the proper small difference 
\begin{align*} 
&\quad \, T^{(2)}( R_{2n}; x_1, x_2, x_3, x_4, y_1, y_2, y_3, y_4) 
-T^{(2)}( Q_{2n}; x_1, x_2, x_3, x_4, y_1, y_2, y_3, y_4) \nonumber \\ 
&=2( x_1y_1-x_1-y_1 ) (x_2y_2-x_2-y_2)(x_3y_4-x_3-y_4)
(x_3-1)^{n-1} (y_4-1)^{n-1}, 
\end{align*} 
where we use $x_3$, $x_4$, $y_3$ and $y_4$ instead of 
$x_{\cap(1,2)}$, $x_{\cup(1,2)}$, $y_{\cap(1,2)}$ and $y_{\cup(1,2)}$ in short.   
\end{rem} 
The calculation is not difficult, but very long. 
Therefore, we leave the proof in the Appendix.

\subsection{Non-isomorphic matroids which need higher genus}\label{sec:4.2}
In this subsection, we always assume that the ground set $E$ of our matroids is the cyclic group 
$\mathbb{Z}_{4n}$, where $n\geq 3$. Moreover, when there is no fear of confusion, 
the cyclic group $\mathbb{Z}_{4n}$ is identified as the set of integers  $\{ z\mid 0\leq z\leq 4n-1\}$. 

\begin{df} 
Let $\mathcal{F}_1$, $\mathcal{F}_2$, $\mathcal{F}_3$, $\mathcal{F}_4$ 
be the families of subsets of $E$ defined by 
\begin{align*} 
\mathcal{F}_1&:=\{ \{ 2i, 2i+1, 2i+2\} \mid 0\leq i\leq n-2\} \cup \{ \{ 2n-2, 2n-1, 0\} \}, \\ 
\mathcal{F}_2&:=\{ \{ 2i, 2i+1, 2i+2\} \mid n\leq i\leq 2n-2\} \cup \{ \{ 4n-2, 4n-1, 2n\} \}, \\ 
\mathcal{F}_3&:=\{ \{ 2i, 2i+1, 2i+2\} \mid 0\leq i\leq n-1\}, \\ 
\mathcal{F}_4&:=\{ \{ 2i, 2i+1, 2i+2\} \mid n\leq i\leq 2n-1\} . 
\end{align*} 
We say that $\mathcal{F}_1$ and  $\mathcal{F}_2$ are \textit{circuital loose cycles} of length $n$, 
and $\mathcal{F}_3$ and  $\mathcal{F}_4$ are \textit{circuital loose paths} of length $n$. 
Let $S_{4n}$ and $S'_{4n}$ be the independent systems as follows; 
\begin{align*} 
\mathcal{B}(S_{4n})&=\mathcal{B}(U_{3, 4n}) \setminus  
(\mathcal{F}_1 \cup \mathcal{F}_2); \\ 
\mathcal{B}(S'_{4n})&=\mathcal{B}(U_{3, 4n})\setminus  
(\mathcal{F}_3 \cup \mathcal{F}_4). 
\end{align*} 
\end{df} 
%
%
\begin{fact} 
Notice that $S'_{4}$ is not a matroid because the independent sets $\{ 0,2\}$ and $\{ 0, 1, 3\}$ force 
one of $\{ 0,1,2\}$ and $\{ 2,3,0\}$ to be independent. 
For the same reason, $S_8$ is not a matroid. 
Therefore, the assumption $n\geq 3$ is essential. 
\end{fact} 

\begin{rem} The odd elements of $E$ are essential for the same reason that $Q_4$ does not satisfy the axioms of independent sets. 
\end{rem} 

\begin{fact} \label{uniqueness_big_circuit}
Let $M$ be one of $S_{4n}$ and $S'_{4n}$, and $y, z\in \ZZ_{4n}$ be elements. 
Then, the number of circuits of three elements containing both $y$ and $z$ is at most one. 
\end{fact} 
\begin{proof} 
This fact can be easily checked. 
\end{proof} 

\begin{prop} 
The independent systems $S_{4n}$ and $S'_{4n}$ are matroids. 
\end{prop}
\begin{proof} 
Let $M$ be one of the independent systems $S_{4n}$ and $S'_{4n}$. 
Let $A_1$, $A_2$ be independent subsets of $M$ on $E$ such that $|A_2|< |A_1|\leq 3$. 
If $|A_2|\leq 1$, then there exists an element $x\in A_1\setminus A_2$ such that  
 $A_2\cup \{ x\}$ is independent because all sets of size at most two are independent 
in both $S_{4n}$ and $S'_{4n}$.  
Therefore, we can assume $|A_2|=2$ and $|A_1|=3$.  
If there does not exist a circuit of size $3$ which contains $A_2$, 
then $A_2\cup \{ x\}$ is independent for each element $x\in A_1\setminus A_2$. 
Therefore, we consider the case that $A_2$ is contained in a circuit of size $3$. 
Let us assume that $\{ z_1, z_2, z_3\}$ is a circuit, and that  $A_2=\{ z_1, z_2\}$. 
Because $A_1$ is independent, $A_1$ contains an element $x\not\in \{ z_1, z_2, z_3\}$. 
Then, $A_2\cup \{ x\}$ is independent from Fact \ref{uniqueness_big_circuit}. 
%
\end{proof} 

\begin{fact} 
If a subset $A\subset E$ is none of $\{ \{ 2n-2, 2n-1, 0\}, \{ 2n-2, 2n-1, 2n\}, \{ 4n-2, 4n-1, 0\} , \{ 4n-2, 4n-1, 2n\}
 \}$, 
then $\rho_{S_{4n}}(A)=\rho_{S_{4n}}(A+2n)=\rho_{S'_{4n}}(A)=\rho_{S'_{4n}}(A+2n)$ holds. 
\end{fact} 
\begin{fact} 
For a subset $A\ni 0$ of $E$, $\rho_{S_{4n}}(A)=\rho_{S'_{4n}}(((A+2n)\setminus \{ 2n\} )\cup \{ 0\} )$ 
\end{fact} 

\begin{thm} \label{thm:existance_higher_genus}
For any $g\in\NN$, there exists a pair of matroids $(M,M')$ such that 
$T^{(g)}(M)=T^{(g)}(M')$ and $M\not\simeq M'$. 

\end{thm} 
\begin{proof} 
For a $g$-tuple $\mathcal{A}=(A_1, A_2, \ldots, A_g)\ (A_i\subset E)$, 
define a family $\mathcal{X}(A)$ as follows; 
\begin{align*} 
\mathcal{X}(\mathcal{A}):=\{ A_i\}_{i=1}^g \cup \{ A_i\cup A_j\}_{1\leq i<j\leq g} 
\cup \{ A_i\cap A_j\}_{1\leq i<j\leq g} . 
\end{align*}  
We fix a  positive integer $n$ such that $n>g^2=|\mathcal{X}(\mathcal{A})|$.


Let $l$ and $l'$  be the non-negative integer  defined as follows; 
$l:=0$ if neither $\{ 2n-2, 2n-1, 0\}$ nor $\{ 2n-2, 2n-1, 2n\}$ belongs to 
$\mathcal{X}(\mathcal{A})$, 
$l:=\min\{ i\mid \{ \{ 2(n-i), 2(n-i)-1, 2(n-i-1)\} \not\in \mathcal{X}(\mathcal{A})$ otherwise; 
$l':=0$ if neither $\{ 4n-2, 4n-1, 0\}$ nor $\{ 4n-2, 4n-1, 2n\}$ belongs to $\mathcal{X}(\mathcal{A})$, 
$l':=\min\{ i\mid \{ \{ 2(2n-i), 2(2n-i)-1, 2(2n-i-1)\} \not\in \mathcal{X}(\mathcal{A})$ otherwise. 
Here, we assume $l\geq l'$ because the other case can also be proved similarly. 
Moreover, the inequation  $n>g^2$ yields $n>l$. 

%


Here,  
we define 
a bijection on $E$ as follows: 
\begin{align*} 
\phi_{\mathcal{A}}(i)
=\begin{cases} 
k &(\textnormal{if } 0\leq k\leq 2(n-l)-1 \textnormal{ or } 2n\leq k\leq 2(2n-l)-1),\\  
k+2n & (\textnormal{if } 2(n-l-1)\leq k\leq 2n-1) \textnormal{ or } 2(2n-l-1)\leq k\leq 4n-1)\\ 
k-2n & (\textnormal{if } 2(2n-l-1)\leq k\leq 4n-1). 
\end{cases}
\end{align*}

Since the map $\phi_\mathcal{A}$ is bijective for any family $\cal{A}$, hence, for any set $R\subset E$, 
we have $|\phi_{\mathcal{A}}(R)|=|R|$. 
Furthermore, for the case that $R$ is none of the four sets $\{ 2n-2, 2n-1, 0\}$, $\{ 2n-2, 2n-1, 2n\}$, $\{ 4n-2, 4n-1, 0\}$, $\{ 4n-2, 4n-1, 2n\}$, the equation $\rho_{S^{\prime}_{4n}}(\phi_{\mathcal{A}}(R))=\rho_{S_{4n}}(R)$ clearly holds. 
For the case that $R$ is one of the four sets $\{ 2n-2, 2n-1, 0\}$, $\{ 2n-2, 2n-1, 2n\}$, $\{ 4n-2, 4n-1, 0\}$, $\{ 4n-2, 4n-1, 2n\}$, we have:
\begin{align*} 
\rho_{S_{4n}}(\{ 2n-2, 2n-1, 0\}) &= 2, & \rho_{S'_{4n}}(\{ 4n-2, 4n-1, 0\}) &= 2, \\ 
\rho_{S_{4n}}(\{ 4n-2, 4n-1, 2n\}) &= 2, & \rho_{S'_{4n}}(\{ 2n-2, 2n-1, 2n\}) &= 2, \\ 
\rho_{S_{4n}}(\{ 2n-2, 2n-1, 2n\}) &= 3, & \rho_{S'_{4n}}(\{ 4n-2, 4n-1, 2n\}) &=3, \\ 
\rho_{S_{4n}}(\{ 4n-2, 4n-1, 0\}) &= 3, & \rho_{S'_{4n}}(\{ 2n-2, 2n-1, 4n\}) &=3.
\end{align*}
Then, for any $i,j$, we have 
\begin{align*} 
\left\{
\begin{array}{l}
|A_i|=|\phi_{\mathcal{A}}(A_i)|, \\ 
\rho_{S_{4n}}(A_i)=\rho_{S'_{2n}}(\phi_{\mathcal{A}}(A_i)), \\ 
|A_i\cap A_j|=|\phi_{\mathcal{A}}(A_i)\cap \phi_{\mathcal{A}}(A_j)|, \\ 
\rho_{S_{4n}}(A_i\cap A_j)=\rho_{S'_{2n}}(\phi_{\mathcal{A}}(A_i)\cap \phi_{\mathcal{A}}(A_i)), \\ 
|A_i\cup A_j|=|\phi_{\mathcal{A}}(A_i)\cup \phi_{\mathcal{A}}(A_j)|, \\ 
\rho_{S_{4n}}(A_i\cup A_j)=\rho_{S'_{2n}}(\phi_{\mathcal{A}}(A_i)\cup \phi_{\mathcal{A}}(A_i)). 
\end{array}
\right.
\end{align*} 


Therefore, we obtain the same monomial in $R^{(g)}(M)$ and $R^{(g)}(M')$. 
\end{proof}  

Suppose that, for any $g$-tuple $\mathcal{A}$, the family $\mathcal{X}(\mathcal{A})$ contains neither a circuital loose cycle nor a loose path of length $n$. 
Then, from Theorem \ref{thm:existance_higher_genus}, 
we have that $T^{(g)}(S_{4n}) = T^{(g)}(S'_{4n})$. 
In other words, if $T^{(g)}(S_{4n}) \neq T^{(g)}(S'_{4n})$ 
holds for some positive integer $g$, then we have a $g$-tuple $\mathcal{A}$ such that the family $\mathcal{X}(\mathcal{A})$ contains either a circuital loose cycle or a loose path of length $n$. 

Using the above observation with some detailed calculations, 
we obtain a considerably better lower bound value as follows:


\begin{prop} \label{prop:g}
If $g \leq \left\lceil \frac{5+ \sqrt{8n-15}}{2}\right\rceil$
then $T^{(g)}(S_{4n}) = T^{(g)}(S'_{4n})$.
\end{prop}


In order to prove the above, we use the following two lemmata.

\begin{lem} 
Let $g$ be a positive integer such that 
$T^{(g)}(S_{4n})\neq T^{(g)}(S'_{4n})$. 
And let $\mathcal{A}=(A_1, A_2, \ldots, A_g)\ (A_i\subset E)$
be a $g$-tuple such that $\mathcal{X}(\mathcal{A})$ contains either a circuital loose cycle or 
loose path of length $n$. 
Suppose that there exist two distinct indices $i\neq j$ such that  $A_i\cup A_j$ is a circuit of size $3$ 
which belongs to  a circuital loose cycle or 
loose path of length $n$. 
Then we have another family 
$\mathcal{A}'=(A'_1, A'_2, \ldots, A'_{g-1})\ (A'_i\subset E)$ such that $\mathcal{X}(\mathcal{A}')$ contains either a circuital loose cycle or 
loose path of length $n$. 
\end{lem}
\begin{proof}
Let $A_i\cup A_j (i\neq j)$ be a circuit of size $3$ which belongs to a circuital loose cycle or loose path of length $n$. 
If $A_i\subset A_j$, then we have 
$\mathcal{A}':=\mathcal{A}\setminus \{ A_i\}$.  
Otherwise, we have 
$\mathcal{A}':=(\mathcal{A}\setminus \{ A_i, A_j\})\cup \{ A_i\cup A_j\}$ 
because none of $A_i$, $A_j$, $A_i\cap A_k$, and
$A_j\cap A_k$ $k\not\in \{ i,j\}$ is a circuit of size $3$. 
\end{proof} 

\begin{lem} 
Let $n$ be a positive integer at least $4$ and 
let $g$ be a positive integer such that 
$T^{(g)}(S_{4n})\neq T^{(g)}(S'_{4n})$. 
And let $\mathcal{A}=(A_1, A_2, \ldots, A_g)$ $(A_i\subset E)$
be a $g$-tuple such that $\mathcal{X}(\mathcal{A})$ contains either a circuital loose cycle or 
loose path of length $n$. 
Then, there exist $\{A'_i \subset E \}_{i=1}^g$ such that 
we can construct either a circuital loose cycle or path of length $n$ by using the set of hyper-edges 
$\{A'_i\cap A'_j\mid 3 \leq i < j \leq g \} \cup \{A'_1, A'_2\}$.  
\end{lem} 

\begin{proof} 
This proof is divided into several subcases as follows: 

\noindent
(Case 1) 
If $\mathcal{F}_1$ does not contain $A_i\cap A_j$ and 
$\{ A_i\}_{i=1}^g=\mathcal{F}$, 
then renumbering of the indices of $\{A_i\}$ we assume that 
$A_1=\{ 0, 1,2\}$, $A_2=\{ 2,3,4\}$, $\ldots$, $A_{n-1}=\{ 2n-4, 2n-3, 2n-2\}$, $A_n=\{ 2n-2, 2n-1, 0\}$. 


Then setting $A'_i:=A_i \cup A_{i+1}=\{ 2i, 2i+1, 2i+2, 2i+3,2i+4\}$ for each $i\in \mathbb{Z}_{n}$, 
we obtain $A'_{i-1}\cap A'_i=\{ 2i, 2i+1, 2i+2\}$. 





\vspace{10pt}

Thus, from now on, we assume that there exists a pair of indices $(i,j) \/ (i\neq j)$ such that $A_i \cap A_j \in \mathcal{F}_1$.

\vspace{10pt}

\noindent
(Case 2) 
Let $l\geq 4$ and let us assume without loss of 
generality that  
$A_1\cap A_2=\{ 0,1,2\}$, $A_3=\{ 2,3,4\}$, 
$A_4=\{ 4,5,6\}$, \ldots, $A_{l}=\{ 2(l-2), 2l-3, 2l-2\}$, 
$A_{l+1}\cap A_{l+2}=\{ 2l-2, 2l-1, 2l\}$. 
Then we redefine that $A_i := A_i\setminus 
\{ z\mid 3\leq z\leq 2l-3\}$ 
for $i\in \{ 1,2,l+1,l+2\}$. 

\noindent
(Subcase 2-1) Let us assume that $l-1=n$, i.e., $A_{l+1}\cap A_{l+2}=\{ 0,1,2\}$. Let us define that 
$A'_1:=\{ z\mid 0\leq z\leq 2n-1\}$ and $A'_i:=A_i$ for $2\leq i\leq l$. 
Then, the family $\mathcal{A}':=\{ A'_i \mid 1\leq i\leq l\}$ satisfies
$\mathcal{F}_1\subset \{ A'_i \cap A'_1\mid i\neq 1\}$. 
\noindent
(Subcase 2-2) 
Let us assume that $A_2\not\in \{ A_{l+1}, A_{l+2}\}$. 
Let us define that 
$A'_1:=A_1$, $A'_i:=A_i \cup A_{i+1}=\{ 2(i-2), 2i-3,2i-2, 2i-1, 2i\}$ for $2\leq i\leq l-1$, 
and $A'_i:=A_i$ for $i\in \{ l, l+1, l+2\}$. 
Then, the family $\mathcal{A}':=\{ A'_i \mid 1\leq i\leq l\}$ satisfies
$\{ \{ 2i, 2i+1, 2i+2\} \mid 0\leq i\leq l-1\}
\subset \{ A'_i \cap A'_j\mid i\neq j\}$. 


\vspace{10pt}

\noindent
(Case 3) 
Let $l$ be an integer at least 3, and we assume that 
$\mathcal{F}_1\subset \{ A_i\cap A_j\mid 1\leq i<j\leq g-l\} 
\cup \{ B_i\mid g-l+1\leq i\leq g\} $ 
and $B_i\cap B_j=\emptyset$ for each $i\neq j$. 
Let us define that $A'_i:=A_i$ for $1\leq \forall i \leq g-l$ and $B'_j:=B_j\cup B_{j+1} \pmod{l}$. 
Then, $B_i\cap B_j\neq \emptyset$ if and only if $i-j\neq \pm 1$. 
Then, the family 
$\mathcal{A'}=\{ A'_i \}_{i=1}^{g-l} \cup \{ B'_j\}_{j= g-l+1}^{g}$ satisfies 
$\mathcal{F}_1\subset \{ A'_i\cap A'_j (i \neq j)\} \cup \{ B'_i\cap B'_j (i \neq j)\}$. 

From the proof so far, 
there exist a family $\{ A_i\}_{i=1}^g$ and an integer $l \in \{0,1,2\}$ such that 
$\mathcal{F}_1\subset  \{ A_i\cap A_j\mid 1\leq i<j\leq g-l\} \cup \{ A_i \mid g-l\leq i\leq g\}$. 

\end{proof} 

\begin{proof} [Proof of Proposition \ref{prop:g}]
Solving the following inequality: $\binom{g-2}{2} +2\geq n$, 
we obtain the desired bound. 
\end{proof} 

Theorem \ref{genus_base_number} says that
$T^{(g)}(S_{4n})\not=T^{(g)}(S_{4n})$ for
 for $g\geq \binom{4n}{3}-2n$. 
 However, this upper bound is too large. 

Here, we introduce an invariant of finite simple graphs. 
\begin{df} 
Let $G=(V(G), E(G))$ be a finite simple graph. 
We say that a family $\mathcal{H}=\{ H_1, \ldots, H_s\}$ of subsets of $V$ is a \textit{$k$-fold intersecting cover} if, for each edge $\{u,v\} \in E(G)$, there exists a $k$-subset $X_{\{u,v \}}$ of the index set $[1,s]$ such that $\{ u, v\}=\cap_{i \in X_{\{u,v\}}} H_i $. 
We say that a finite simple $G$ is \textit{$k$-fold intersecting coverable} by $s$ sets if there exists a family $\mathcal{H}=\{ H_1, H_2, \ldots, H_s\}$ of subsets of $V$ such that $\mathcal{H}$ is a $k$-fold intersecting cover of $G$. 
Let $\iota_k(G)$ be the minimum number $s$ such that $G$ is $k$-fold intersecting coverable by $s$ sets. 
We call this number $\iota_k(G)$ the $k$-fold intersecting cover number of $G$. For simplicity, only when $k=2$, we use the symbol $\iota(G)$ instead of $\iota_2(G)$.
Clearly, $G$ has a $2$-fold intersecting cover with $|E(G)|+1$ sets, those are the vertex set and all edges. 
\end{df} 


Let $C_n$ denote a cycle graph whose edge set is 
$\{\{i,i+1\} \mid i\in\mathbb{Z}_n\}$, 
and $P_{n+1}$ denote a path graph 
whose edge set is $\{\{i,i+1\} \mid 0 \leq i \leq n \}$.
\begin{thm} 
The equation $\iota(C_n)=\iota(P_{n+1})$ holds. 
\end{thm} 
\begin{proof} 
Let $\mathcal{H}=\{ H_i\}_{i=1}^s$ be a $2$-fold intersecting cover of $P_{n+1}$, 
and $\psi$ be the ring isomorphism from $\ZZ$ to $\ZZ_n$. 
For a set $H_i\in \mathcal{H}$, let us define $H'_i:=\{ \psi(z) \mid z\in H\}$. 
Note that $|H'_i\cap H'_j|\leq |H_i\cap H_j|$ for each $i, j$. 
Then, if $H_i\cap H_j$ is an edge of $P_{n+1}$, $H'_i\cap H'_j$ is an edge of $C_n$. 
Therefore, $\mathcal{H}':=\{ H'_i\}_{i=1}^s$ is  a $2$-fold intersecting cover of $C_n$. 

Conversely, let $\mathcal{H}=\{ H_i\}_{i=1}^s$ be a $2$-fold intersecting cover of $C_n$ 
such that $\{ 0, 1\}=H_1\cap H_s$, and $\psi'$ be the natural injection 
from $\mathbb{Z}_n$ to $\{0, \ldots, n\}$ such that 
$\psi'(\bar{0})=0$ and $\psi'(\bar{i})=i$ for each $\bar{i}\in \mathbb{Z}_n \setminus \{\bar{0}\}$. 
Then, we define the set $H'_i$ as follows; 
(i) if $\{ \bar{n-1}, \bar{0}, \bar{1} \} \subset H_i$, let $H'_i:=\{ \psi'(\bar{z}) \mid \bar{z}\in H_i\} \cup \{ n\}$; 
(ii) if $\{ \bar{n-1}, \bar{0} \} \subset H_i\not\ni \bar{1}$, 
 let $H'_i:=(\{ \psi'(\bar{z}) \mid \bar{z}\in H_i\} \cup \{ n\})\setminus \{ 0\}$;   
 (iii) otherwise, $H'_i:=\{ \psi'(\bar{z}) \mid \bar{z}\in H_i\}$. 
 Then, if $H_i\cap H_j$ is an edge of $C_n$, $H'_i\cap H'_j$ is an edge of $P_{n+1}$. 
 Therefore, $\mathcal{H}':=\{ H'_i\}_{i=1}^s$ is  a $2$-fold intersecting cover of $P_{n+1}$. 
\end{proof} 
\begin{thm} \label{thm:iota_inf} 
The inequation
 $T^{\iota(C_n)}(S_{4n})\not=T^{\iota(C_n)}(S'_{4n})$ holds. 
\end{thm} 
\begin{proof} 
The Whitney rank generating function $R^{\iota(C_n)}(S_{4n})$ has a monomial 
generated by $\mathcal{F}_1$, while $R^{\iota(C_n)}(S'_{4n})$ 
does not. 
\end{proof}

\begin{thm} \label{thm:iota_sup} 
For each integer at least $3$, we have $\iota (C_n)\leq 2\lceil \sqrt{n}\rceil$. 
\end{thm} 
\begin{proof} 
(Case 1) We consider the case that $n$ is a square number at least $9$, e.g.,  we assume  $n=p^2$. 
Then, we have $p\geq 3$. 

Take $2p$ sets $A_0$, $A_1$, \ldots, $A_{p-1}$, $B_0$, $B_1$, \ldots, $B_{p-1}$ as follows;  
\begin{align*} 
A_i&:=\{ z\in \mathbb{Z}_{p^2} \mid i p \leq z\leq  i p+1\}, \\ 
B_j 
&:=\bigcup_{0\leq k\leq p-2} \left( (j+k+1 )p+\{ k+1,k+2\} \right) 
\end{align*} 
Because $A_{i+1}=A_i+p$ and $B_{j+1}=B_j+p$ hold for $0\leq i, j\leq p-1$, 
the indices are taken modulo  $p$.  

It is clear that for $0\leq i\leq p-1$, we have 
\begin{align} 
A_i\cap A_{i+1}=\{( i+1)p, (i+1)p+1\}.  \label{equation_As}
\end{align} 
It is also clear that $A_0\cap B_0=\{ p+1, 0\}$ from the definition 
$B_0=\{ p+1, p+2, , 2p+2, 2p+3, 3p+3, 3p+4, \ldots, (p-1)p+p-1, (p-1)p+p=0\}$. 
Thus, we have $A_{i}\cap B_i=\{ ip, (i+1)p+1\}$ from $A_i=A_{i-1}+ip$ and $B_i=B_{i-1}+ip$. 

%

For $1\leq r\leq p-1$,
we have the following by substituting $j=q-p$ and $k=r-1(\leq p-2)$;  
\begin{align*} 
A_q\cap B_{q-r} &\supset \left( ((q-r)+(r-1)+1 )p+\{ (r-1)+1, (r-1)+2\} \right) \\ 
&=\{ pq+r, pq+r+1\}.  
\end{align*}
If $r\not=p-1$, we have $\{ p(q+1)+r+1, p(q+1)+r+2\} \subset A_{q+1} \cap B_{q-r}$ 
and $\{ p(q+1)+r+1, p(q+1)+r+2\}\cap (A_q\cap A_{q+1})=\emptyset$ from $r\geq 1$ and (\ref{equation_As}).  
If $r=p-1$, we have $B_{q-r}=B_{q+1}$, i.e., $A_{q+1}\cap B_{q+1}=\{ (q+1)p, (q+2)p+1\}$ 
and $A_q\cap B_{q-r}=\{ (q+1)p-1, (q+1)p\}$. 
Then, we have $p(q+2)+1\not\in A_q$ from (\ref{equation_As}).  
In both cases, the elements of  $( A_{q+1}\cap B_{q-r})  \setminus \{ pq+r, pq+r+1\}$ do not belong to $A_q$. 
By analogous argument,  
the elements of  $( A_{q-1}\cap B_{q-r})  \setminus \{ pq+r, pq+r+1\}$ do not belong to $A_q$. 
Therefore,  we have the following equation;  
\begin{align} 
A_q\cap B_{q-r}=
\begin{cases} 
\{ pq, p(q+1)+1\} &(\textnormal{if } r=0), \\ 
\{ pq+r, pq+r+1\} & (\textnormal{otherwise }). \label{equation_between_AB} 
\end{cases} 
\end{align} 
Therefore, $p+1$ edges of $A_q$ are 
$A_{q-1}\cap A_q=\{ pq, pq+1\}$, $A_q\cap B_{q-1}=\{ pq+1, pq+2\}$, $A_q\cap B_{q-2}=\{ pq+2, pq+3\}$, 
\ldots, $A_q\cap B_{q+1}=\{p(q+1)-1, p(q+1)\}$ $A_q\cap A_{q-1}=\{ p(q+1), p(q+1)+1\}$. 


(Case 2) We consider the case that $n$ is not a square number at least $12$. 
Let $p$  be the smallest integer such that $n<p^2$. Then we have $p\geq 4$. 
 Let $\mathcal{H}^{(0)}=\{ A^{(0)}_i, B^{(0)}_j \mid i, j\in \mathbb{Z}_p\}
:=\{ A_i, B_j \mid i, j\in \mathbb{Z}_p\}$ be a $2$-fold intersecting cover of $\ZZ_{p^2}$ in (Case 1). 
Define graphs $G^{(q)}$ ($1\leq q\leq p$) with vertex set $\ZZ_{p^2}$ recursively as follows;  
$G^{(0)}:=C_{p^2}$, take an integer $k_q$ such that $2\leq k_q\leq p-1$, and 
define $G^{(q)}$ as 
$E(G^{(q)})=(E(G^{(q-1)} ) \setminus \{ (i, i+1)\mid pq+1\leq i\leq pq+k_q-1\})\cup \{ (pq+1, pq+k_q)\}$. 
Clearly, we have $|E(G^{(q)})|=|E(G^{(q-1)})|-(k_q-2)$. 
In particular, $|E(G^{(q)})|=|E(G^{(q-1)})|-(p-3)$ if $k_q=p-1$. 
For each $1\leq q\leq p$, 
we define recursively  $\mathcal{H}^{(q)}=\{ A^{(q)}_i, B^{(q)}_{q-j} \mid i, j\in \mathbb{Z}_p\}$ as follows; 
\begin{align*} 
A^{(q)}_i&:=
\begin{cases} 
A^{(q-1)}_i & (\textnormal{if } i\not=q),  \\ 
A^{(q-1)}_i\setminus \{ pq+2, pq+3, \ldots, pq+k_q-1 \} & (\textnormal{if } i=q), 
\end{cases} \\ 
B^{(q)}_{q-j}&:=
\begin{cases} 
B^{(q-1)}_{q-j} & (\textnormal{if } k_q \leq j\leq p ), \\ 
B^{(q-1)}_{q-j}\setminus \{ pq+j, pq+j+1\} & (\textnormal{if } 2\leq j\leq k_q-1), \\ 
(B^{(q-1)}_{q-j}\setminus \{ pq+2\})\cup \{ pq+k_q\} &(\textnormal{if }j=1). 
\end{cases} 
\end{align*}

Then, the following holds for each $1\leq q\leq p$;

\begin{itemize}
\item $pq+2, pq+3, \ldots, pq+k_q \in A^{(q-1)}_q\setminus 
\bigcup_{i\not= q} A^{(q-1)}_i$  because of  $2\leq k_q \leq p-1$; 
\item $B^{(q)}_j\cap A^{(q)}_i=B^{(q-1)}_j\cap A^{(q-1)}_i$ holds for each $j$ if $i\not=q$
 because each element of $\mathbb{Z}_{p^2}\setminus A^{(q-1)}_q$ belongs to $A^{(q)}_i$ (resp. $B^{(q)}_j$)  
if and only if it belongs to $A^{(q-1)}_i$ (resp. $B^{(q-1)}_j$); 
\item $A^{(q)}_{i-1}\cap A^{(q)}_i=A^{(q-1)}_{i-1}\cap A^{(q-1)}_i=\{ip, ip+1\}$ also holds for each $i$; 
\item the edges in $A^{(q)}_q$ of $G^{(q)}$ are  
$A^{(q)}_{q-1}\cap A^{(q)}_q=\{ pq, pq+1\}$, $A^{(q)}_q\cap B^{(q)}_{q-1}=\{ pq+1, pk+k_q\}$, 
$A^{(q)}_q\cap B^{(q)}_{q-k_q}=\{ pq+k_q, pq+k_q+1\}$, 
\ldots, $A^{(q)}_q\cap B^{(q)}_{q+1}=\{ p(q+1)-1, p(q+1)\}$, and $A^{(q)}_q\cap A^{(q)}_{q+1}=\{ p(q+1), p(q+1)+1\}$. 
\end{itemize} 
Then, $\mathcal{H}^{(q)}$ is a $2$-fold intersecting cover of $G^{(q)}$. 
Therefore, $C_n$ has a $2$-fold intersecting cover with $2p$ sets if $3p\leq n<p^2$.

(Case 3) We consider the remaining small $n$. 
Then, the $2$-fold intersecting covers 
\begin{itemize} 
\item $\{ \{ 1, 2\}, \{ 2, 0\}, \{ 0, 1\}, \{ 0,1,2\} \}$ for the case of $C_3$, 
\item $\{ \{ i, i+1, i+2\} \mid i\in \mathbb{Z}_4\} $, for the case of $C_4$, 
\item $\{ \{ i, i+1, i+2\} \mid i\in \mathbb{Z}_5\} $ for the case of $C_5$, 
\item $\{ \{ i, i+1, i+2\} \mid i\in \mathbb{Z}_6\} $ for the case of $C_6$, 
\item $A_1=\{0,1,2\}, A_2=\{ 1,2,3,4\}, A_3=\{ 3,4,5,6\}, A_4=\{ 5,6,0,1\}, A_5=\{ 2,3,6,0\}, A_6=\{ 4,5\} $ for $C_7$, 
\item $A_1=\{ 0,1,2,3\}, A_2=\{ 2,3,4,5\}, A_3=\{ 4,5,6,7\}, A_4=\{ 6,7,0,1\}, A_5=\{ 1,2,5,6\}, A_6=\{ 3,4,7,0\} $ for $C_8$, 
\item $A_1=\{ 0,1,2,3\}, A_2=\{ 2,3,4,5\}, A_3=\{ 4,5,6,7\}, A_4=\{ 6,7,8,9\}, 
A_5=\{ 8,9,0,1\}, A_6=\{ 1,2,5,6\}, A_7=\{ 3,4,7,8\}, A_8=\{ 9,0\} $ for $C_{10}$, 
\item $A_1=\{ 0,1,2,3,4\}, A_2=\{ 3,4,5,6\}, A_3=\{ 5,6,7,8\}, A_4=\{ 7,8,9,10\}, 
A_5=\{ 9,10,0,1\}, A_6=\{ 1,2,6,7\}, A_7=\{ 2,3,8,9\}, A_8=\{ 4,5,10,0\} $ for the case $C_{11}$
\end{itemize} 
satisfy the statement of our theorem. 
\end{proof} 

\begin{prop} \label{prop:triangle_square} 
For each integer 
 $n\geq 3$, we have the following; 
\begin{align*} 
2\left\lceil \sqrt{n} \right\rceil \leq \left\lceil \sqrt{2} 
\left\lceil \frac{5+\sqrt{8(n-2)+1}}{2}\right\rceil -2 \right\rceil. 
\end{align*} 
\end{prop} 
Let $u_n$, $l_n$, $\hat{l}_n$ be the following monotonically non-decreasing sequences such that 
\begin{align*} 
u_n&= 2\left\lceil \sqrt{n} \right\rceil,  \\ 
l_n&=\left\lceil \frac{-1+\sqrt{8(n-2)+1}}{2}+3\right\rceil, \\ 
\hat{l}_n&=\left\lceil \sqrt{2} \, l_n \right\rceil-2. 
\end{align*} 
\begin{proof} 
If $n=3$, it is trivial because $u_3=\hat{l}_3=4$.   
Let us assume that $n\geq 4$. 
Let $n'$ be the maximum square number at most $n$. 
 It is well-known that  
$(l_{n'}-3)(l_{n'}-2)/2$ is the minimum triangular number at least $n'-2$. 
Therefore, we have 
\begin{align*} 
4n' \leq 2(l_{n'}-3)(l_{n'}-2)+8 
=\left(\sqrt{2} l_{n'}-\frac{5}{\sqrt{2}}\right)^2-\frac{5}{2} 
<\left\lceil\sqrt{2}l_{n'}-3 \right\rceil ^2. 
\end{align*} 
Because the both sides of the above inequation are square numbers, 
we have 
\begin{align*} 
u_{n'}=2\sqrt{n'}<\left\lceil \sqrt{2}l_{n'}\right\rceil -3. 
\end{align*} 
Because the both sides of the above inequation are positive integers, we have 
\begin{align*} 
u_{n'}\leq \left\lceil \sqrt{2}l_{n'}\right\rceil -4=\hat{l}_{n'}-2.  
\end{align*} 
If $n=n'$, we have done. 
If $n>n'$, then 
we have the following inequation; 
\begin{align*} 
u_n=u_{n'+1}=u_{n'}+2\leq \hat{l}_{n'}\leq \hat{l}_n. \qedhere
\end{align*} 
\end{proof}

Here, we reprise Theorem \ref{thm:main} (2). 
\begin{thm} \label{matroidsMandN}
For any positive integer $g$, there exist two 
matroids $M$ and $N$ such that
$T^{(g)}(M)=T^{(g)}(N)$ and 
$T^{(\lceil\sqrt{2}g  \rceil)}(M) \neq T^{(\lceil\sqrt{2}g \rceil)}(N)$.
\end{thm}
\begin{proof} 
It is clear from Theorem \ref{thm:iota_inf}, Theorem \ref{thm:iota_sup}, and Proposition \ref{prop:triangle_square}.   
\end{proof}

\section{Concluding remarks}\label{sec:rem}

\subsection{Special values of $T^{(2)}$} 
\begin{thm}\label{thm:special}
We have 
\begin{align*}
T(M;x,y)
&=T^{(2)}(M;2,2,0,x,2,2,0,y)\\
&=T^{(2)}(M;2,2,x,0,2,2,y,0)\\
&=T^{(2)}(M;2,2,0,x,2,2,y,0).
\end{align*}
\end{thm}
\begin{proof}
We show the first identity
$T(M;x,y)
=T^{(2)}(M;2,2,0,x,2,2,0,y)$. 
The other cases can be proved similarly. 
The right-hand side is written as follows:
\begin{align*}
T^{(2)}&(M;2,2,0,x,2,2,0,y)\\
&=\sum_{A_1,A_2\subset E}
(-1)^{\rho E-\rho (A_1\cap A_2)}
(x-1)^{\rho E-\rho (A_1\cup A_2)}
(-1)^{ |A_1\cap A_2|-\rho (A_1\cap A_2)}
(y-1)^{ |A_1\cup A_2|-\rho (A_1\cup A_2)}\\
&=
\sum_{A_1,A_2\subset E}
(-1)^{\rho E- |A_1\cap A_2|}
(x-1)^{\rho E-\rho (A_1\cup A_2)}
(y-1)^{ |A_1\cup A_2|-\rho (A_1\cup A_2)}. 
\end{align*}
We claim that if we fix $A_1\neq \emptyset$ then
\begin{align*}
\sum_{A_1,A_2\subset E}
(-1)^{\rho E- |A_1\cap A_2|}
(x-1)^{\rho E-\rho (A_1\cup A_2)}
(y-1)^{ |A_1\cup A_2|-\rho (A_1\cup A_2)}=0. 
\end{align*}
In fact, 
if $|A_1|$ is odd then for fixed $A_2$ the terms $(A_1,A_2)$ and 
$(A_1,A_1\cup A_2)$ have opposite signature since 
\begin{align*}
|A_1\cup A_2|\not\equiv |A_1\cup (A_1\cup A_2)| \pmod{2}. 
\end{align*}
If $|A_1|$ is even then for fixed $A_2$ 
\begin{align*}
\sum_{S\subset A_1}
(-1)^{\rho E- |A_1\cap (S\cup A_2)|}
(x-1)^{\rho E-\rho (A_1\cup (S\cup A_2))}
(y-1)^{ |A_1\cup (S\cup A_2)|-\rho (A_1\cup (S\cup A_2))}=0. 
\end{align*}
It is because 
if $S_1$ is odd and $S_2$ is even 
$(A_1,S_1\cup A_2)$ and $(A_1,S_2\cup A_2)$ have opposite signature since 
\begin{align*}
|A_1\cap (S_1\cup A_2)|\not\equiv |A_1\cap (S_2\cup A_2)| \pmod{2}. 
\end{align*}
This completes the proof of Theorem \ref{thm:special}. 

\end{proof}
By Theorem \ref{thm:Greene}, we have the following 
corollary: 
\begin{cor}
Let $M$ be a vector matroid over $\FF_q$, and $C_M$ be the corresponding code. 
Then 
\begin{align*}
w_{C_M}(x,y)
&=T^{(2)}\left(M;2,2,0,\frac{x+(q-1)y}{x-y},2,2,0,\frac{x}{y}\right)\\
&=T^{(2)}\left(M;2,2,\frac{x+(q-1)y}{x-y},0,2,2,\frac{x}{y},0\right)\\
&=T^{(2)}\left(M;2,2,0,\frac{x+(q-1)y}{x-y},2,2,\frac{x}{y},0\right). 
\end{align*}

\end{cor}
Using $(6)\sim (8)$ in {\cite[p.268]{Welsh}}, we have the following 
corollary: 
\begin{cor}
Let $M$ be a vector matroid. 
Then 
\begin{align*}
&\hspace{-30pt}\mbox{number of basis of $M$}\\
&=T^{(2)}(M;2,2,0,1,2,2,0,1)\\
&=T^{(2)}(M;2,2,1,0,2,2,1,0)\\
&=T^{(2)}(M;2,2,0,1,2,2,1,0), \\
&\hspace{-30pt}\mbox{number of independent sets of $M$}\\
&=T^{(2)}(M;2,2,0,2,2,2,0,1)\\
&=T^{(2)}(M;2,2,2,0,2,2,1,0)\\
&=T^{(2)}(M;2,2,0,2,2,2,1,0), \\
&\hspace{-30pt}\mbox{number of spanning sets of $M$}\\
&=T^{(2)}(M;2,2,0,1,2,2,0,2)\\
&=T^{(2)}(M;2,2,1,0,2,2,2,0)\\
&=T^{(2)}(M;2,2,0,1,2,2,2,0).
\end{align*}

\end{cor}

\subsection{Witt's problem from the viewpoint of the Tutte polynomials of genus $g$}

In coding theory, 
it is called Witt's problem that for the Type II codes 
of length $8n$, to determine the minimum number $g$ 
that the genus $g$ weight enumerators of those codes 
are linearly independent. 
For the case $n=1$, there is a unique Type II code. Therefore, 
the crucial problems are the cases $n\geq 2$. 
The answer is known for the cases $n=2,3,$ and $4$. 
For $n=2,3,$ and $4$, the answer is $3,6,$ and $10$, respectively (\cite{Nebe}). 

In this section, 
we study Witt's problem from the viewpoint of Tutte polynomial of genus $g$, 
namely to determine the minimum number $g$ that the genus $g$ Tutte polynomial 
of those codes are linearly independent. 
The results are as follows: 

\begin{thm}
For $n=2$, 
the genus one Tutte polynomial $T^{(1)}$
of Type II codes of length $8n$ are linearly independent. 
For $n=3$, 
the genus one Tutte polynomial $T^{(1)}$
of Type II codes of length $8n$ are linearly dependent. 
\end{thm}
\begin{proof}
We perform brute-force enumeration based 
on the definition by using 
{\sc Magma} \cite{Magma} and 
{\sc Mathematica} \cite{Mathematica}. 
The data is given in \cite{codes} and \cite{Miezaki}. 
\end{proof}

\begin{rem}
\begin{itemize}

\item 
The classification of Type II codes of length $8n$ is known for $n\leq 5$. 
It is an interesting problem to determine whether the genus g Tutte polynomials of those codes are linearly independent.


\item 

In \cite{Greene}, 
a relationship between 
the weight enumerators of codes and the Tutte polynomials of matroids was established. 
In this remark, 
we review this relationship to explain the term \textit{genus}. 

Let $M$ be a vector matroid obtained from the $k\times n$ matrix $A$. 
Then the row space of $A$ is an $[n,k]$ code over $\FF_q$, namely 
a $k$-dimensional subspace of of $\FF_q^n$. We denote such a code by $C_M$. 
The weight enumerator $w_C(x,y)$ of the code $C$ is the homogeneous polynomial
\[
w_C(x,y) = 
\sum_{c\in C}
x^{n-\wt(c)} y^{wt(c)} =
\sum_{i=0}^n A_i x^{n-i}y^i,
\]
where $A_i=\sharp\{i\mid c=(c_1,\ldots,c_n)\in C,c_i\neq 0\}$. 
The Tutte polynomial of a vector matroid $M$ 
and the weight enumerator of $C_M$ have the following relation:
\begin{thm}[\cite{Greene}]\label{thm:Greene}
Let $M$ be a vector matroid on a set $E=\{1,\ldots,n\}$ over $\FF_q$. 
Then
\[
w_{M_C}(x_1,x_2)=
x_2^{n-\dim(M_C)}(x_1-x_2)^{\dim(M_C)}
T\left(M_C; \frac{x_1+(q-1)x_2}{x_1-x_2},\frac{x_1}{x_2}\right). 
\]
\end{thm}
A generalization of the weight enumerator is known as 
the weight enumerator of genus $g$: 
\[
w_C^{(g)}(x_a:a\in \FF_2^g)=\sum_{v_1,\ldots,v_g\in C}
\prod_{a\in \FF_2^g}x_a^{n_a(v_1,\ldots,v_g)}, 
\]
where $n_a(v_1,\ldots,v_g)$ denotes the number 
of $i$ such that $a=(v_{1i},\ldots,v_{gi})$. 
This gives rise to a natural question: is there
a generalization of the Tutte polynomial that relates 
the complete weight enumerator $w_C^{(g)}(x_a:a\in \FF_2^g)$?
We believe that the Tutte polynomials of genus $g$ is 
a candidate generalization that answers this. 
Therefore, we use the term \textit{genus}. 
So far, we do not have any relations between 
the complete weight enumerator of genus $g$ and 
the Tutte polynomials of genus $g$. 

\end{itemize}
\end{rem}

\subsection{Future problems}

We conclude this paper by stating a few open problems.

\begin{enumerate}
\item 

Following the same policy as in the proof of the main theorem of this paper,
we believe that the following conjecture can also be proved.
\begin{conj}
Let $M$ be a non-separable matroid, let $C$ be a largest circuit of $M$, and
let $D$ be a largest cocircuit of $M$.
Then $M$ can be reconstructed from its Tutte
polynomial of genus $|\mathcal{B}(M)|-\max\{|C|,|D|\}+3$. 
\end{conj}


  \item 
  \noindent
  \begin{prob}
        {Given an arbitrary matroid $M$, does the upper bound on the genus
  of the Tutte polynomial of $M$ needed to reconstruct $M$ belong to $\ord(|B(M)|)$? }
  \end{prob}
  \item
  In 2000, Bollob\'{a}s, Pebody and Riordan \cite{BPR} conjectured that almost
  all graphic matroids are $T^{(1)}$-unique. Related to this conjecture, let us ask
  the following question:
  \begin{prob}
  Are all graphical matroids $T^{(2)}$-unique? More generally, 
  does there exist a positive integer $N$ such that 
  all graphical matroids are $T^{(N)}$-unique or of class $T^{(N)}$? 
\end{prob}

\item
As a direct consequence of Theorem \ref{matroidsMandN}, we obtain the following proposition:
\begin{prop}
For any positive integer $g$, there exist two positive numbers $c_1$, $c_2$ and two infinite sequences 
matroids $\{ M_i\}$ and $\{ N_i\}$ such that
$T^{(g)}(M_i)=T^{(g)}(N_i)$ and 
$T^{(\lceil c_1 g +c_2 \rceil)}(M_i) \neq T^{(\lceil c_1 g +c_2\rceil)}(N_i)$.
\end{prop}
\noindent
In connection with this proposition, we pose the next problem:
\begin{prob}
    {Find two infinite sequences of matroids $\{ M_i\}$ and $\{ N_i\}$ such that the constant $c_1$ is as small as possible.}
\end{prob}
\noindent
More precisely, we have the following conjecture: 
\begin{conj}
There exist two infinite sequences of
matroids $\{ M_i\}$ and $\{ N_i\}$ such that
$T^{(g)}(M_i)=T^{(g)}(N_i)$ and 
$T^{( g + 1)}(M_i) \neq T^{(  g +1)}(N_i)$.
\end{conj}

\end{enumerate}

\appendix
\section{The difference $T^{(2)}(R_{2n})-T^{(2)}(Q_{2n})$} 
Here, we calculate The difference $T^{(2)}(R_{2n})-T^{(2)}(Q_{2n})$, 
For a pair $(A_1, A_2)$ of subsets of the ground set $E$ of the matroid $M$, 
we define $f(A_1, A_2)$ as 
\begin{align*} 
f(A_1, A_2)&:=(x_1-1)^{\rho(E)-\rho(A_1)} (y_1-1)^{|A_1|-\rho(A_1)} 
(x_2-1)^{\rho(E)-\rho(A_2)} \nonumber \\ &\times (y_2-1)^{|A_2|-\rho(A_2)} 
(x_3-1)^{\rho(E)-\rho(A_1\cap A_2)} 
(y_3-1)^{|A_1\cap A_2|-\rho(A_1\cap A_2)} \nonumber \\ 
&\times (x_4-1)^{\rho(E)-\rho(A_1\cup A_2)} (y_4-1)^{|A_1\cup A_2|-\rho(A_1\cup A_2)}, 
\end{align*} 
\paragraph{} 

For a matroid $M$ on $E$,  we define the following polynomials; 
\begin{align} 
h(M; s, t, u)&:=
\sum_{ \substack{A_1, A_2\subset E,\\ |A_1\cap A_2|=s \\ |A_1|=t,\\ |A_2|=u}} 
f(A_1, A_2)
\label{partition_by_three_parameters},  
\end{align}

From (\ref{definition_genus_Tutte}) and (\ref{partition_by_three_parameters}), 
it is clear that the following equation holds; 
\begin{align*} 
T^{(2)} (M; \{x_1, x_2, x_3, x_4\}, \{y_1, y_2, y_3, y_4\})
= \sum_{(s,t,u)\in (E \cup \{ 0\})^3} h(M; s,t,u). 
\end{align*} 

Since both $R_{2n}$ and $Q_{2n}$ are self-dual, we have 
\begin{align} 
h(R_{2n}; s, t, u)&=h(R_{2n}; t+u-s, 2n-t, 2n-u), \label{self-dual-R_2n}\\ 
 h(Q_{2n}; s, t, u)&=h(Q_{2n}; t+u-s, 2n-t, 2n-u).  \label{self-dual-Q_2n}
\end{align}  
\begin{prop} \label{n_free} 
If there does not exist $n$ among $s, t, u, t+u-s$, 
then we have $h(R_{2n}; s, t, u)=h(Q_{2n}; s, t, u)$.  
\end{prop} 
\begin{proof} 
If there does not exist $n$ among $s, t, u, t+u-s$, 
then we have 
\begin{align*} 
f(A_1, A_2)&=(x_1-1)^{\max \{ n-t, 0\} } (y_1-1)^{\max \{ t-n, 0\} } 
(x_2-1)^{\max \{ n-u, 0\} } \nonumber \\ &\times (y_2-1)^{\max \{ u-n, 0\}} 
(x_3-1)^{\max \{ n-s, 0\} } (y_3-1)^{ max \{ s-n, 0\} } \nonumber \\ 
&\times (x_4-1)^{ \max \{ n+s-t-u, 0\}} (y_4-1)^{\max \{ t+u-s-n, 0\} } 
\end{align*} 
holds for any pair $(A_1, A_2)$. 
Therefore, we have 
\begin{align*} 
&\textcolor{white}{=}h(R_{2n}; s, t, u)\\ 
&=h(Q_{2n}; s, t, u) \\ 
&=\displaystyle\binom{2n}{s}\displaystyle\binom{2n-s}{t-s}\displaystyle\binom{2n-t}{u-s}
(x_1-1)^{\max \{ n-t, 0\} } (y_1-1)^{\max \{ t-n, 0\} } 
 \nonumber \\ &\times (x_2-1)^{\max \{ n-u, 0\} } (y_2-1)^{\max \{ u-n, 0\}} 
(x_3-1)^{\max \{ n-s, 0\} } (y_3-1)^{ max \{ s-n, 0\} } \nonumber \\ 
&\times (x_4-1)^{ \max \{ n+s-t-u, 0\}} (y_4-1)^{\max \{ t+u-s-n, 0\} }.  
\qedhere
\end{align*} 
\end{proof} 
\begin{prop} \label{cupn} 
If the equation $t+u-s=n$ holds, then 
the equation $h(R_{2n}; s, t, u)=h(Q_{2n}; s, t, u)$ holds.  
\end{prop} 
\begin{proof} Note that there exist $\binom{2n}{n}$ choice of $A_1\cup A_2$. \\ 
(Case 1) Suppose that $\max \{ t, u\}\leq n-1$ holds. 
For $R_{2n}$, we have 
\begin{align*} 
&f(A_1, A_2) 
= 
\begin{cases} 
(x_1-1)^{n-t} 
(x_2-1)^{ n-u } 
(x_3-1)^{ n-s } \\ 
\times (x_4-1) (y_4-1) & \textnormal{if $A_1\cup A_2\in X_1, X_2$}\\ 
(x_1-1)^{n-t} 
(x_2-1)^{ n-u } 
(x_3-1)^{ n-s } & \textnormal{otherwise.}
\end{cases} 
\end{align*} 
For $Q_{2n}$, we have 
\begin{align*} 
&f(A_1, A_2) 
= 
\begin{cases} 
(x_1-1)^{n-t} 
(x_2-1)^{ n-u } 
(x_3-1)^{ n-s } \\ 
\times (x_4-1) (y_4-1) & \textnormal{if $A_1\cup A_2\in X_1, X_3$}\\ 
(x_1-1)^{n-t} 
(x_2-1)^{ n-u } 
(x_3-1)^{ n-s } & \textnormal{otherwise.}
\end{cases} 
\end{align*} 
For each $A_1\cup A_2$ of size $n$, there exist 
$\displaystyle\binom{n}{t}\binom{t}{u}$ pairs of $(A_1, A_2)$. 
Therefore, 
\begin{align*} 
&\textcolor{white}{=}h(R_{2n}; s, t, u)\\ 
&=h(Q_{2n}; s, t, u) \\ 
&=\binom{n}{t}\binom{t}{u}\left( 2(x_4-1) (y_4-1) +\binom{2n}{n} -2\right)  \\
&\times (x_1-1)^{n-t} 
(x_2-1)^{ n-u } 
(x_3-1)^{ n-s } . 
\end{align*} 
(Case 2) Suppose that $\min\{ t, u\}<\max \{ t, u\}=n$ holds. 
Without loss of generality, we can assume $t=n$. 
For $R_{2n}$, we have 
\begin{align*} 
&f(A_1, A_2) 
= 
\begin{cases} 
(x_1-1)(y_1-1) 
(x_2-1)^{ n-u } 
(x_3-1)^{ n-s } \\ 
\times (x_4-1) (y_4-1) \quad \qquad \textnormal{if $A_1\cup A_2\in \{ X_1, X_2\}$}\\ 
(x_2-1)^{ n-u } 
(x_3-1)^{ n-s } \quad \textnormal{otherwise.}
\end{cases} 
\end{align*} 
For $Q_{2n}$, we have 
\begin{align*} 
&f(A_1, A_2) 
= 
\begin{cases} 
(x_1-1)(y_1-1) 
(x_2-1)^{ n-u } 
(x_3-1)^{ n-s } \\ 
\times (x_4-1) (y_4-1) \quad \qquad \textnormal{if $A_1\cup A_2\in \{ X_1, X_3\}$}\\ 
(x_2-1)^{ n-u } 
(x_3-1)^{ n-s } \quad \textnormal{otherwise.}
\end{cases} 
\end{align*} 
For each $A_1=A_1\cup A_2$ of size $n$, there exist 
$\displaystyle\binom{n}{u}$ choices of $A_2=A_1\cap A_2$. 
Therefore, we have 
\begin{align*} 
&\textcolor{white}{=}h(R_{2n}; s, t, u)\\ 
&=h(Q_{2n}; s, t, u) \\ 
&=\binom{n}{u}\left( 2(x_1-1)(y_1-1)(x_4-1) (y_4-1) +\binom{2n}{n} -2\right)  \\
&\times (x_2-1)^{ n-u } 
(x_3-1)^{ n-s } . 
\end{align*} 
(Case 3) Let us assume that $t=u=t+u-s=n$ holds. 
Then, $s=n$ also holds. 
Therefore, we have 
\begin{align*} 
&\textcolor{white}{=}h(R_{2n}; s, t, u)\\ 
&=h(Q_{2n}; s, t, u) \\ 
&=2\prod^4_{i=1} (x_i-1)\prod^4_{i=1} (y_i-1). \qedhere
\end{align*} 
\end{proof} 
\begin{prop} \label{capn} 
For any pair of $t$ and $u$, the equation $h(R_{2n}; n, t, u)=h(Q_{2n}; n, t, u)$ holds.  
\end{prop} 
\begin{proof} 
From the equations (\ref{self-dual-R_2n}), (\ref{self-dual-Q_2n}) and 
Proposition \ref{cupn}, we have 
\begin{align*} 
h(R_{2n}; n, t, u) 
&=h(R_{2n}; t+u-n, 2n-t, 2n-u) \\ 
&=h(Q_{2n}; t+u-n, 2n-t, 2n-u)\\
&=h(Q_{2n}; n, t, u). \qedhere 
\end{align*} 
\end{proof} 
\begin{prop} \label{t=n} 
For all $(s, t, u)$ such that an either $t$ and $u$ is $n$, 
and that the equality $s\not=n$ and $t+u-s\not=n$ hold,  
we have $h(R_{2n}; s, t, u)=h(Q_{2n}; s, t,u)$. 
\end{prop} 
\begin{proof} 
Without loss of generality, we can assume $t=n$ and $u\not=n$. 
At first, we show (\ref{rank_of_A_2}), (\ref{rank_of_cap}), and (\ref{rank_of_cup}) hold 
for each pair  $(A_1, A_2)$ such that all of 
\begin{align} 
&|A_1|=t=n, &|A_2|=t\not=n, 
&|A_1\cap A_2|=s\not=n, 
&|A_1\cup A_2|=t+u-s\not=n \label{assumption_of_t=n} 
\end{align} 
 hold.  
From $t\not=n$, we have 
\begin{align} 
\rho(A_2)=\min \{ t, n\}. \label{rank_of_A_2}
\end{align} 
From $|A_1\cap A_2|=s \leq n=t=|A_1|$ and the assumption $s\not=n$, 
we have $|A_1\cap A_2|\leq n-1$, i.e., we have 
\begin{align} 
\rho(A_1\cap A_2)=s. \label{rank_of_cap}
\end{align}  
From $|A_1\cup A_2|=t+u-s \geq n=t=|A_1|$ and the assumption $t+u-s\not=n$, 
we have $|A_1\cup A_2|\geq n+1$, i.e., we have 
\begin{align} 
\rho(A_1\cup A_2)=n. \label{rank_of_cup} 
\end{align}  

Note that there are $\displaystyle\binom{2n}{n}\binom{n}{s}\binom{n}{u-s}$ pairs of 
$(A_1, A_2)$ such that (\ref{assumption_of_t=n}) holds  
by choosing $n$ elements of $A_1$ from $E$ at first, 
choosing $s$ elements of $A_1\cap A_2$ from $A_1$ secondly, 
and choosing $u-s$ elements of $A_2\setminus A_1$ from $E\setminus A_1$. 
Among such pairs, 
$2\displaystyle\binom{n}{s}\binom{n}{u-s}$ pairs satisfy $\rho(A_1)=n-1$ 
for both $R_{2n}$ and $Q_{2n}$. 
Therefore, we have 
\begin{align*} 
&\textcolor{white}{=}h(R_{2n}; s, n, u) \\ 
&=h(Q_{2n}; s, n, u) \\ 
&= 2 \binom{n}{s}\binom{n}{u-s} 
(x_1-1)(y_1-1) (x_2-1)^{\max \{ 0, n-u\} } 
(y_2-1)^{\max \{ 0, u-n \}} \\ 
&\times (x_3-1)^{ \max \{ 0, n-s\} } 
(y_3-1)^{\max \{ 0, s-n\} } 
(y_4-1)^{u-s} \\ 
&+ \left( \binom{2n}{n}-2\right) \binom{n}{s}\binom{n}{u-s} 
 (x_2-1)^{\max \{ 0, n-u\} } 
(y_2-1)^{\max \{ 0, u-n \}} \\ 
&\times (x_3-1)^{ \max \{ 0, n-s\} } 
(y_3-1)^{\max \{ 0, s-n\} } 
(y_4-1)^{u-s}. \qedhere  
\end{align*} 
\end{proof} 

\begin{prop} \label{s_n_n} 
For any integer $s$ at least 2, 
the equality $h(R_{2n}, s, n, n)=h(Q_{2n}, s, n, n)$ holds. 
\end{prop} 
\begin{proof} 
From Proposition \ref{cupn} and Proposition \ref{capn}, 
we can assume that $|A_1\cap A_2|=s\leq n-1$ and 
$|A_1\cup A_2|=2n-s\geq n+1$, i.e., 
$\rho(A_1\cap A_2)=s$ and $\rho(A_1\cup A_2)=n$ holds 
for both cases of $R_{2n}$ and $Q_{2n}$. 

From $X_1\cap X_2=\emptyset$ and $X_1\cap X_3=\{ n\}$, 
there does not a pair $(A_1, A_2)$ such that 
$\rho(A_1)=\rho(A_2)=n-1$. 
Note that there exist $\displaystyle\binom{2n}{n}\binom{n}{s}\binom{n}{n-s}$ 
pairs of $(A_1, A_2)$ such that $|A_1|=|A_2|=n$ and that $|A_1\cap A_2|=s$. 
Among those pairs, $\rho(A_1)=\rho(A_2)+1=n$ holds if and only if 
\begin{align*} 
\begin{cases} 
A_2\in \{ X_1, X_2\}  & \textnormal{ for the case $R_6$, }\\ 
A_2\in \{ X_1, X_3\} & \textnormal{ for the case $Q_6$. } 
\end{cases} 
\end{align*} 
The number of pairs of $(A_1, A_2)$ with $\rho(A_1)=\rho(A_2)+1=n$ is 
$2\displaystyle\binom{n}{s}\binom{n}{n-s}$ for both cases of $R_{2n}$ and $Q_{2n}$. 
There are also $2\displaystyle\binom{n}{s}\binom{n}{n-s}$ pairs of $(A_1, A_2)$ 
such that $\rho(A_2)=\rho(A_1)+1=n$ for both cases of $R_{2n}$ and $Q_{2n}$.  
The other $\left( \displaystyle\binom{2n}{s}-4\right) 
\displaystyle\binom{2n-s}{n-s}\binom{n}{n-s}$ 
pairs of $(A_1, A_2)$ satisfies  $\rho(A_1)=\rho(A_2)=n$. 
Therefore, for $s\geq 2$, we have 
\begin{align*} 
&\textcolor{white}{=}h(R_{2n}, s, n, n) \\ 
&=h(Q_{2n}, s, n, n) \\ 
&=2\binom{n}{s}\binom{n}{n-s} 
(x_1-1)(y_1-1)  (x_3-1)^{n-s} (y_4-1)^{n-s} \\ 
&+ 2\binom{n}{s}\binom{n}{n-s} 
(x_2-1)(y_2-1)  (x_3-1)^{n-s} (y_4-1)^{n-s} \\ 
&+\left( \binom{2n}{s}-
4 \right)\binom{2n-s}{n-s}\binom{n}{n-s} (x_3-1)^{n-s} (y_4-1)^{n-s}. \qedhere 
\end{align*} 
\end{proof} 
From Propositions \ref{n_free}, \ref{cupn}, \ref{capn}, \ref{t=n}, and \ref{s_n_n},  
\begin{align} 
&T^{(2)}( R_{2n}, x_1, x_2, x_3, x_4, y_1, y_2, y_3, y_4) 
-T^{(2)}( Q_{2n}, x_1, x_2, x_3, x_4, y_1, y_2, y_3, y_4) \nonumber \\ 
&=h( R_{2n}; 0, n, n)+h(R_{2n}; 1, n,n) 
-h( Q_{2n}; 0, n, n)-h(Q_{2n}; 1, n,n) \label{2_n_n} 
\end{align} 
holds for each integer $n$ at least 3. 

From here, we calculate $h( R_{2n}; 0, n, n)$, $h(R_{2n}; 1, n,n)$,  
$h( Q_{2n}; 0, n, n)$, and $\textcolor{blue}{h}(Q_{2n}; 1, n,n)$.  \\ 
Note that there exist $\displaystyle\binom{2n}{n}\binom{n}{1}\binom{n}{1}
=n^2\displaystyle\binom{2n}{n}$ 
pairs of $(A_1, A_2)$ such that $|A_1|=|A_2|=n$ and $|A_1\cap A_2|=1$ hold 
by choosing $n$ elements of $A_1$ from $E$ first, 
and choosing a unique element of $A_1\cap A_2$ from $A_1$
secondly, and choosing a unique element of $E\setminus (A_1\cup A_2)$ from 
$E\setminus A_1$ lastly for both cases of $R_{2n}$ and $Q_{2n}$. 

In the matroid $R_{2n}$, there does not exist a pair $(A_1, A_2)$ such that 
$\rho(A_1)=\rho(A_2)=n$ as same as the condition $s\geq 2$. 
Therefore, we have 
\begin{align} 
&\textcolor{white}{=}h(R_{2n}, 1, n, n) \nonumber \\ 
&=2n^2 \left( 
(x_1-1)(y_1-1)+ (x_2-1)(y_2-1)\right) (x_3-1)^{n-1} (y_4-1)^{n-1} 
\nonumber \\ 
&+\left( \binom{2n}{n}-4 \right)n^2(x_3-1)^{n-1} (y_4-1)^{n-1}. \label{R_1_n_n} 
\end{align} 

In the matroid $Q_{2n}$, there exist two pairs of $(A_1, A_2)$ such that 
$\rho(A_1)=\rho(A_2)=|A_1|-1=|A_2|-1=n-1$ and $|A_1\cap A_2|=1$, 
that is the case of $\{A_1, A_2\}=\{ X_1, X_3\}$. 
We count the number of $(A_1, A_2)$ such that 
$\rho(A_1)=\rho(A_2)-1=|A_1|-1=|A_2|-1=n-1$ and $|A_1\cap A_2|=1$. 
Let us assume $A_1=X_1$. Then, there exist $\displaystyle\binom{n}{1}\binom{n}{1}=n^2$ 
sets such that both $|A_2|=n$ and $|A_1\cap A_2|=1$ hold. 
Among these $n^2$ choices of $A_2$, $n^2-1$ sets of $A_2$ other than 
the case $A_2=X_3$ satisfies $\rho(A_2)=n$. 
There are $n^2-1$ choices of $A_2$ such that $A_1=X_3$. 
There exist $2(n^2-1)$ pairs of $(A_1, A_2)$ such that 
$\rho(A_1)=\rho(A_2)-1=|A_1|-1=|A_2|-1=n-1$ and $|A_1\cap A_2|=1$. 
There also exist $2(n^2-1)$ pairs of $(A_1, A_2)$ such that 
$\rho(A_1)-1=\rho(A_2)=|A_1|-1=|A_2|-1=n-1$ and $|A_1\cap A_2|=1$. 
The other $n^2\displaystyle\binom{2n}{n}-2-4(n^2-1)$ pairs of 
$(A_1, A_2)$ satisfy $\rho(A_1)=\rho(A_2)=|A_1|=|A_2|=n$. 
Therefore, we have 
\begin{align} 
&\textcolor{white}{=} h(Q_{2n}; 1, n, n) \nonumber \\ 
&=2(x_1-1)(y_1-1)(x_2-1)(y_2-1)(x_3-1)^{n-1}(y_4-1)^{n-1} \nonumber \\ 
&+2(n^2-1) (x_1-1)(y_1-1)(x_3-1)^{n-1}(y_4-1)^{n-1} \nonumber \\ 
&+2(n^2-1) (x_2-1)(y_2-1)(x_3-1)^{n-1}(y_4-1)^{n-1} \nonumber \\ 
&+\left( \binom{2n}{n}n^2-4n^2+2\right) (x_3-1)^{n-1}(y_4-1)^{n-1} . 
\label{Q_1_n_n}
\end{align} 

There are $\displaystyle\binom{2n}{n}$ pairs of $(A_1, A_2)$ such that 
$|A_1\cap A_2|=0$, $|A_1|=|A_2|=n$ for both cases of $R_{2n}$ and $Q_{2n}$. 

In the matroid $R_{2n}$, there exist two pairs of $(A_1, A_2)$ such that  
$\rho(A_1)=\rho(A_2)=|A_1|-1=|A_2|-1=n-1$ and $A_1\cap A_2=\emptyset$, 
that is the case of $\{A_1, A_2\}=\{ X_1, X_2\}$. 

The condition $A_1\not\in \{ X_1, X_2\}$ and $A_2=\bar{A_1}$ 
yields $A_2\not\in \{ X_1, X_2\}$. 
Then, $\displaystyle\binom{2n}{n}-2$ pairs of $(A_1, A_2)$ 
with $A_1\cap A_2=\emptyset$ satisfy  
$\rho(A_1)=\rho(A_2)=|A_1|=|A_2|=n$. 
Therefore, we have 
\begin{align} 
&\textcolor{white}{=}h(R_{2n}; 0, n, n) \nonumber \\ 
&=
\left( 2(x_1-1) (y_1-1) (x_2-1) (y_2-1) +\binom{2n}{n}-2\right) (x_3-1)^n (y_4-1)^n.  
\label{R_0_n_n} 
\end{align} 

In the matroid $Q_{2n}$, there exist two pairs of $(A_1, A_2)$ such that 
$\rho(A_1)=\rho(A_2)-1=|A_1|-1=|A_2|-1=n-1$ and $A_1\cap A_2|=\emptyset$, 
that is the case of $A_1\in \{ X_1, X_3\}$. 
There also exist two pairs of $(A_1, A_2)$ such that 
$\rho(A_1)-1=\rho(A_2)=|A_1|-1=|A_2|-1=n-1$ and $A_1\cap A_2=\emptyset$. 
The other $\displaystyle\binom{2n}{n}-4$ pairs  of $(A_1, A_2)$ 
with $A_1\cap A_2=\emptyset$ satisfy 
$\rho(A_1)=\rho(A_2)=|A_1|=|A_2|=n$. 
Therefore, we have 
\begin{align}
&\textcolor{white}{=}h(Q_{2n}; 0, n, n) \nonumber \\ 
&=\left( 2(x_1-1) (y_1-1)+2(x_2-1) (y_2-1) + 
\binom{2n}{n}-4\right)  (x_3-1)^n (y_4-1)^n. \label{Q_0_n_n}  
\end{align} 
From (\ref{2_n_n}), (\ref{R_1_n_n}), (\ref{Q_1_n_n}), (\ref{R_0_n_n}), and (\ref{Q_0_n_n}), 
we have 
\begin{align*} 
&T^{(2)}( R_{2n}; x_1, x_2, x_3, x_4, y_1, y_2, y_3, y_4) 
-T^{(2)}( Q_{2n}; x_1, x_2, x_3, x_4, y_1, y_2, y_3, y_4) \nonumber \\ 
&=2( x_1y_1-x_1-y_1 ) (x_2y_2-x_2-y_2)(x_3y_4-x_3-y_4)
(x_3-1)^{n-1} (y_4-1)^{n-1}. 
\end{align*} 


\section*{Acknowledgements}
The authors are supported by JSPS KAKENHI (22K03277, 22K03398).



\end{document}